\documentclass[final]{siamltex}
\usepackage{amsmath, amsfonts,color}
\usepackage[english]{babel}
\usepackage{dsfont}
\usepackage{amssymb}
\usepackage{mathabx}
\usepackage{enumitem}
\usepackage{graphicx,epsfig}
\usepackage{float}
\usepackage{color}
\usepackage{capt-of}
\usepackage{tikz}

\newcommand{\R}{{\mathbb R}}
\newcommand{\N}{{\mathbb N}}

\newcommand{\one}{{\mathds{1}}}

\newcommand{\cF}{{\mathcal F}}

\newcommand{\cM}{{\mathcal M}}
\newcommand{\cW}{{\mathcal W}}

\let\ds=\displaystyle

\def\R{{\mathbb R}}
\def\N{{\mathbb  N}}
\def\N{{\mathbb N}}
\def\ds{\displaystyle}

 \definecolor{mypurple}{RGB}{140,0,255}
\definecolor{myred}{rgb}{255,0,0}

\definecolor{mydarkturquoise}{RGB}{0,206,209}
\definecolor{mydeeppink}{RGB}{255,20,147}
\definecolor{darkblue}{RGB}{0,0,140}
\definecolor{blue2}{RGB}{0,0,0}
\definecolor{middleblue}{RGB}{0,0,71}
\definecolor{light-gray}{gray}{0.9}

\definecolor{ProcessBlue}{cmyk}{1,0,0,0.40}
\definecolor{Black}{cmyk}{0,0,0,1}
\definecolor{Red}{cmyk}{0,1,1,0.2}
\definecolor{Green}{cmyk}{0.9,0,1,0}
\definecolor{Orange}{cmyk}{0,0.61,0.87,0.5}
\definecolor{Fuchsia}{cmyk}{0.47,0.91,0,0.06}
\definecolor{PineGreen}{cmyk}{0.92,0,0.59,0.30}

\numberwithin{equation}{section}
 \def\theequation{\thesection.\arabic{equation}}
 \newtheorem{remark}{Remark}[section]
 \newtheorem{assumption}{Assumption}[section]

\DeclareMathOperator*{\argmin}{arg\,min}

\begin{document}

\title{A mean field model for the interactions between firms on the markets of their inputs}

\author{
  Yves Achdou \thanks { Universit{\'e} de Paris Cité and  Sorbonne Universit{\'e}, CNRS, Laboratoire Jacques-Louis Lions, (LJLL), F-75006 Paris, France, achdou@ljll-univ-paris-diderot.fr}
  \and
Guillaume Carlier\thanks{CEREMADE, Universit\'e Paris
Dauphine, PSL, Pl. de Lattre de Tassigny, 75775 Paris Cedex 16, FRANCE and INRIA-Paris, MOKAPLAN,
carlier@ceremade.dauphine.fr},
\and
Quentin Petit \thanks{CEREMADE, Universit\'e Paris
Dauphine, PSL, and EDF R$\&$D, quentin.petit@edf.fr}
\and
Daniela Tonon \thanks{Dipartimento di Matematica "Tullio Levi-Civita", Universit\`a degli Studi di Padova, via Trieste 63, 35121 Padova, Italy.
tonon@math.unipd.it}
  }

\maketitle
\begin{abstract}
  We consider an economy made of competing firms which are heterogeneous in their capital and  use several inputs 
  for producing goods.  Their consumption policy is fixed rationally  by  maximizing a utility and their capital cannot fall below a given threshold
  (state constraint).
  We aim at modeling the interactions between firms on the markets of the different inputs on the long term. 
The stationary equlibria are described by  a system of coupled non-linear differential equations:
 a Hamilton-Jacobi equation describing the optimal control problem of a single atomistic firm;
 a continuity equation describing  the distribution of the  individual state variable  (the capital) in the population of firms; 
 the equilibria on the  markets of the production factors. We prove the existence of equilibria under suitable assumptions.
\end{abstract}

\section{Introduction}
\label{sec:introduction}
We consider an economy made of competing firms which are heterogeneous in their capital, and    use several inputs 
for producing goods.   These inputs, or factors of production, may include raw materials, energy, manpower, rented surface, etc...
We aim at modeling the interactions of the firms on the markets of the different inputs on the long term.
We make the following general assumptions:
\begin{itemize}
\item the economy is reduced to one sector of activity with a large number (in fact a continuum) of  firms competing on the markets of inputs
\item  the firms choose which amount of their capital is invested into production and which amount is consumed 
(for retributing the owners). Their consumption policy is fixed rationally  by  maximizing a utility 
\item the firms are identical in the sense that  (1) two different firms  with the same capital and  quantities of inputs  produce the same amounts of goods
  (2) they have the same utility function
\item there is a state constraint: the capital of any  firm must not fall below a given threshold, fixed to  $0$ in the whole paper
\item for a given firm, all the others are indistinguishable and the firms interact  only via the prices of the different inputs
\item a single firm has a negligible impact on the markets
\item equilibrium on the markets is reached when supply matches  aggregate demand. Supply is assumed to be a given function of prices.
\item closure and creation of firms may happen.  This  will be modeled in what follows.
\end{itemize}

Because we are interested in long term tendencies, we aim at finding stationary equilibria. The outputs of our model will be  
\begin{itemize}
\item the distribution of  capital 
\item the optimal investment/consumption policy of the firms given their capital
\item the unit prices of the different inputs 
\end{itemize}
Our model falls into the wide class of mean field games. 
The theory of  mean field games ({\sl MFGs} for short), has been introduced and studied 
in the pioneering  works of J-M. Lasry and P-L. Lions~\cite{PLL-CDF,MR2269875,MR2271747,MR2295621},
 and aims at studying deterministic or stochastic differential
games (Nash equilibria) as the number of agents tends to infinity. 
It supposes that the rational agents are  indistinguishable and  individually have a negligible influence on the game, 
and that each individual strategy is influenced by some averages of quantities depending on the states (or the controls) of the other agents.
 The applications of MFGs are numerous, from economics to the study of crowd motion. For useful reference on mean field games, one can see for example \cite{MR2762362,cardaliaguet2010, MR4214773,MR3268061}.

Our model will be  summarized by a system of coupled non-linear differential equations:
 (1) a Hamilton-Jacobi-Bellman equation describing the optimal control problem of a single atomistic firm;
 (2) a continuity equation describing  the distribution of the  individual state variable  (the capital) in the population of firms; 
 (3) the equilibria on the  markets of the production factors.

 The present model has some similarities with the time continuous Aiyagari-Bewley-Huggett models \cite{bewley,aiyagari, huggett}
 studied in \cite{MR3268061,MR4365976}. In particular, they all lead to a better understanding  of the
 individual accumulation of capital/investment policy. In the present paper, a key aspect for proving the existence of equilibria is the regularity properties of the individual optimal policies.

The paper is organized as follows: the model, the main results and important examples are presented in  Section \ref{sec:secGeneralities:model}. The mathematical results concerning the optimal control problem of a single firm given the prices of inputs are proved in Section \ref{sec:optim-contr-probl-1}. As already mentioned, the stress will be put on regularity properties of the solutions, which will play an important role in the remaining part of the paper. Then, the distribution of capital among the firms given the prices of inputs will be studied in Section \ref{sec:secGeneralities:fokk-planck-equat}: in particular, we will prove that under the assumptions made, the distribution is absolutely continuous with respect to Lebesgue measure. Finally, in Section \ref{sec:equilibrium},  we use  Brouwer topological degree in order  to obtain  the existence of equilibria.

For keeping the length of the paper reasonable, we have chosen not to discuss the numerical simulations that we have carried out for a model with two factors of production: the manpower and the surface rented by the firms. We refer to \cite{Petit2022phd} for a description of these simulations, a discussion of the results and comparisons with available statistics.

\section{The model and the main results}\label{sec:secGeneralities:model}
In what follow, we give more details  and   write down the different equations which summarize our model. 
First, in paragraph \ref{sec:optim-contr-probl},  we address the strategy of a single firm given the prices of the inputs. 
Second,  in paragraph \ref{sec:distrib_capital},  we propose a model for the distribution of capital,  supposing again that the prices of the inputs are given.
From the two steps above, we can deduce the aggregate demand for the different production factors.
Finally, the model is closed by matching the aggregate demand with the exogenous supply of production factors.
In the three steps mentioned above, we make some assumptions which allow us to prove the existence of a mean field equilibrium.
In subsection \ref{sec:import-exampl-util} below, we give examples in which these assumptions are satisfied.

In the following, we set $\R_+=[0,+\infty)$.

\subsection{The optimal control problem of a single firm given the prices of  inputs}
\label{sec:optim-contr-probl}
The output of a given firm is  $F(k,\ell)\in \R_+$, where $k\in \R_+$ and $\ell  \in \R_+^d$   respectively stand for the capital of the firm and
 for the quantities of the different inputs it uses. The function $F : \R_+ \times \R_+^d\to \R_+$ is the production function. 

Let $w\in (0,+\infty)^d$ be the collection of the unit prices of the different factors of production: depending on $i\in \{1,\dots,d\}$,
  $w_i$ may stand for the unit price of a raw material, the annual salary of a class of workers, the rental price of a unit of surface. The benefits of the firm in a unit of time are therefore given by $F(k,\ell)-w\cdot \ell -\delta k$, where $\delta \ge 0$ is the rate of depreciation of  capital.

 The dynamics of the capital of a given firm is  described by 
\begin{equation}
	\label{eq:chap:MFG_model:secGeneralities:capitalDynamic}
	\frac{dk}{dt}(t) = F(k(t),\ell(t))-w\cdot \ell(t) -\delta k(t)-c(t),
\end{equation}
where $c(t)$ stands for the consumption  at time $t$,  for example the share of the benefits that goes to the owners of the firm. The firm has two variables of control,  its consumption $c(t)\in \R_+$ and the quantities of inputs  $\ell(t)\in \R_+^d$.  

  The firms  face the problem of how to split their benefits between consumption and investments that produce growth. 
A given firm determines its policy by  maximizing  the following payoff:
\begin{equation}
   \label{eq:chap:MFG_model:secGeneralities:2}
	\int_0^{+\infty} U(c(t))e^{-\rho t}dt, 
\end{equation}
where $U:[0,+\infty)\rightarrow [-\infty, +\infty) $ is a  utility function and $\rho$ is a positive discount factor. 
\\
It aims at finding the controls $t\mapsto c(t)\in [0,+\infty)$ and $t\mapsto \ell(t)\in [0,+\infty)^d$
 which maximize  \eqref{eq:chap:MFG_model:secGeneralities:2}, under the constraint that its capital stay nonnegative (state constraint).
 The value of the optimal control problem when the firm has a capital $k_0\ge 0$ is 
\begin{equation}
  \label{eq:chap:MFG_model:secGeneralities:valueFunction}
    \begin{split}
  u(k_0,w)\quad =\quad &\sup_{c, \; \ell, \; k  } \int_0^{+\infty} U(c(t))e^{-\rho t}dt \\
  & \hbox{ subject to }\\
  &
  \left\{ 
    \begin{array}[c]{l}
      c   \in   L^1_{{\rm loc}}( \R_+; \R_+), \quad  \ell\in  L^1_{{\rm loc}}(\R_+; \R_+^d),\quad k\in W^{1,1}_{\rm{loc}} ( \R_+),\\   
      k \hbox{ satisfies  \eqref{eq:chap:MFG_model:secGeneralities:capitalDynamic} for a.a. $t>0$,}\\
     k(0)=k_0,\\
     k(t)\ge 0 \hbox{ for all }t. 
    \end{array}\right.
    \end{split}
\end{equation}
We will see that under suitable assumptions, namely Assumptions \ref{ass:secHJ:1} and \ref{ass:secHJ:3} below, $u(k_0,w)\in \R$  for all $k_0\in (0,+\infty)$.

We expect that the value function $u$ can be found by solving a  Hamilton-Jacobi equation in $(0,+\infty)$ 
with state constraints at $k=0$ (from the dynamic programming principle).
Let the Hamiltonian  $H:\R_+\times\R \times (0,+\infty)^d \rightarrow  (-\infty, +\infty]$ be defined as follows: for all $k\in \R_+$ and $q\in \R$, 
\begin{eqnarray}
  \label{eq:chap:MFG_model:sec:general:hamiltonian}
	H(k,q,w)&=&\ds \sup_{c\in \R_+,\; \ell\in \R_+^d}\left\{U(c) +q\left(F(k,\ell)-w\cdot \ell-\delta k-c\right)\right\}\\
  \label{eq:chap:MFG_model:secGeneralities:32}
   &=&\ds \sup_{c\in \R_+}\left\{U(c)    -cq\right\} +f(k,w)q,
\end{eqnarray}
where $f:\R_+\times (0,+\infty)^d      \rightarrow \R$ is the {\sl net output} function:
\begin{equation}
  \label{eq:chap:MFG_model:secGeneralities:33}
  f(k,w)=\sup_{\ell\in \R_+^d}\left\{F(k,\ell)-w\cdot \ell\right\}-\delta k.
\end{equation}
\begin{remark}
  By contrast with simpler applications of mean field games to price formation, see e.g. \cite{MR4215224},
the Hamiltonian of the problem does not involve a quantity which depends separately/additively on the price vector $w$ and on $q$.
\end{remark}

The Hamilton-Jacobi equation then reads:
\begin{equation}
	\label{eq:chap:MFG_model:secGeneralities:HJ}
	-\rho u(k,w) + H\left(k,\frac {\partial u}{\partial k} (k,w),w\right)=0,\quad \quad \hbox{ in } (0,+\infty).
\end{equation}

Recall that a function $\psi:\R_+^m \to \R$ is monotone if and only if for every  $z,\tilde z \in \R_+^m$,
  \begin{displaymath}
    z \le \tilde z \quad\quad \Rightarrow \quad \quad \psi(z)\leq \psi(\tilde z),
  \end{displaymath}
  where  the partial order  $\;\le\; $ on $\R^{m}$ is defined as follows: 
  \begin{displaymath}
    z\le \tilde z \quad \hbox{ if and only if }  \quad   z_i\leq \tilde z _i,\;\; \forall i\in \{1,\dots,m\}.
  .
  \end{displaymath}

We make the following assumptions on $U$ and $F$:

\medskip

\begin{assumption}[Assumptions on $U$] \label{ass:secHJ:1}
The utility   $ U:\R_+\rightarrow [-\infty, +\infty)$ has the following properties:
\begin{itemize}
\item[i)] $U$ is $C^2$ on $(0,+\infty)$.
\item[ii)] $U$ is increasing and strictly concave in $(0,+\infty)$.
\item[iii)] $\ds \lim_{c\rightarrow 0^+}U'(c) = +\infty$ and $\ds \lim_{c\rightarrow +\infty}U'(c) = 0$.
\end{itemize}
\end{assumption}

\medskip

\begin{assumption}[Assumptions on $F$]  \label{ass:secHJ:3}
  The function $F$ is concave and monotone.
For any vector $w\in (0,+\infty)^d$ and for any  $k\in \R_+$, the  {\sl net output} $f(k,w)$ defined by (\ref{eq:chap:MFG_model:secGeneralities:33}) is finite and
achieved by a unique $\ell =\ell^* (k,w)\in \R_+^d$, and  $\ell^*$  is a  $C^1$ function defined on $(0,+\infty)\times (0,+\infty)^d$.
\\ Moreover,
\begin{enumerate}
\item The function $f$ belongs to $C^0(\R_+\times (0,+\infty)^d)\cap C^1(  (0,+\infty)^{d+1}) $
\item for all $w\in  (0,+\infty)^d$,
  $f (\cdot, w):\R_+\rightarrow \R$ has the following properties:
  \begin{enumerate}
  \item[i)] $f(\cdot, w)$ is locally of class $C^{1,1}$ on $(0,+\infty)$
  \item[ii)] $f(0,w)\geq 0$, $k\mapsto f(k,w)$ is strictly concave and $\lim_{k\rightarrow 0^+} \frac{\partial f }{\partial k} (k,w) = +\infty$
  \item[iii)]  $\lim_{k\rightarrow +\infty} \frac{\partial f }{\partial k} (k,w) =-\delta $
  \end{enumerate}
\end{enumerate}
\end{assumption}

\begin{remark}\label{rem:1}
  \begin{itemize}
  \item From point 2.ii) in Assumption \ref{ass:secHJ:3}, $f(\cdot, w)$ is strictly concave. Hence,
 $ \frac{\partial f }{\partial k} (k,w)$ has a limit as $k\to+\infty$, which belongs to $[-\infty, +\infty)$.
Therefore, point 2.iii)  in Assumption \ref{ass:secHJ:3} is meaningful.
\item If $\delta =0$,
  then the strict concavity of $k\mapsto f(k,w)$ implies that it
 is increasing in $(0,+\infty)$. Then,  because  $f(0,w)\geq 0$,   $f(k,w)>0$  for all $k>0$  
and has a limit   as $k\to +\infty$, which belongs to  $(0,+\infty]$.  
\item If $\delta>0$,
  then $\lim_{k\rightarrow +\infty}f(k,w)=-\infty$, and 
$f$ is negative for $k$ large enough.
  \end{itemize}
\end{remark}

\medskip

\begin{remark}\label{rem:2}
It is clear that $-f$ is monotone with respect to $w$. The optimal quantity of the input labeled $i$ is
  \begin{equation*}
  \ell_i^*(k,w)= - \frac{\partial f }{\partial w_i} (k,w).    
  \end{equation*}
\end{remark}

\medskip



 In Section \ref{sec:optim-contr-probl-1} below, we are going to prove the following theorem:
\begin{theorem}\label{th:secHJ:main}
  Under Assumptions \ref{ass:secHJ:1} and \ref{ass:secHJ:3}, for all $w\in (0,+\infty)^d$,
  there exists a unique classical solution $u(\cdot,w)\in C^1(0,+\infty)$ of \eqref{eq:chap:MFG_model:secGeneralities:HJ} with the following property:
 there  exists a critical value $\kappa^*(w) >0$,  such that
 \begin{eqnarray}
  \label{eq:chap:MFG_model:secHJ:2}
  H_q\left(k,\frac {\partial u}{\partial k} (k,w),w\right)> 0,\quad   & \hbox{for } &0<k< \kappa^*(w),\\
  \label{eq:chap:MFG_model:secHJ:3}
  H_q\left(k,\frac {\partial u}{\partial k} (k,w),w\right)< 0,\quad   & \hbox{for } &\kappa^*(w) < k <+\infty.
\end{eqnarray}
 Here $H_q$ stands for the partial derivative of $H$ with respect to its second argument.\\
Moreover $\kappa^*(w)$ is characterized by the equation 
\begin{equation}
  \label{eq:1}
 \frac{\partial f}{\partial k}(\kappa^*,w)=\rho.
\end{equation}
The function $u(\cdot, w)$ is strictly concave on $(0,+\infty)$ and belongs to $C^2(  (0,\kappa^*(w))\cup (\kappa^*(w), +\infty))$.\\
Furthermore, $u(\cdot,w)$  is the value function of the optimal control problem (\ref{eq:chap:MFG_model:secGeneralities:valueFunction}), and $H_q(k,\partial_k u(k,w),w)$  
is the optimal investment policy of a firm with capital $k$.
\end{theorem}

\begin{remark}
  \begin{enumerate}
  \item The existence of $\kappa^*(w)>0$ such that the
    capital of all firms converges towards $\kappa^*(w)$ is known as the golden rule
    of investment \cite[Chapter 7]{Allais1947}.
	\item We will see in Section \ref{sec:secGeneralities:fokk-planck-equat} below that  a firm with an initial capital $k_0\not=\kappa^*(w)$ never reaches this target capital $\kappa^*(w)$.
\end{enumerate}
\end{remark}

The difficulty in the proof of Theorem \ref{th:secHJ:main} lies in the fact that the Hamiltonian $H(k,q,w)$ is defined only for nonnegative values of $q$ (i.e. $H(k,q,w)=+\infty$ if $q<0$) 
and may blow up as $q\to 0_+$. Hence, classical results on viscosity solutions of Hamilton-Jacobi equations for state constrained optimal control problems cannot be applied in a straightforward manner. We will use a different strategy: in particular, in the simplest case in which $\delta = -\lim_{k\to \infty}   \frac {\partial f}{\partial k} (k,w)=0$, our proof of existence  is based on the fact that the function $q\mapsto H(k,q,w)$ is strictly convex, strictly decreasing in $(0,q_{\min})$ and strictly increasing in $(q_{\min}, +\infty)$, where  $q_{\min}=U'(f(k,w))$, see Lemma \ref{lem:secHJ:Hmin} below. In this case,  our strategy consists in solving two ordinary differential equations by means of shooting methods: the first (resp. second) one involves the  inverse of the increasing (resp. decreasing) part of $ q\mapsto H(k,q,w)$.

Note that a different strategy has been studied in \cite{Petit2022phd}; it was inspired by
a method proposed in \cite{Santambrogio2017} for studying Ramsey model of optimal growth with non local externalities. It consists in introducing a relaxed lagrangian version of the original optimal control problem, then obtaining compactness properties which lead to the existence of an optimal control and of a solution of the original problem.
However, this approach needs an  assumption stronger than Assumption  \ref{ass:secHJ:3}.

\subsection{The distribution of capital  given the prices of inputs}
\label{sec:distrib_capital}
The distribution of capital corresponding to the optimal investment policy of the firms is  a bounded positive measure on $(0,+\infty)$.
 In our model, its density  is characterized by the following continuity equation: 
\begin{equation}
  \label{eq:chap:MFG_model:FP}
  \frac{\partial}{\partial k}\left(m(\cdot,w) H_q\left(\cdot,\frac {\partial u}{\partial k} (\cdot,w),w\right)\right) = \eta(\cdot, u(\cdot,w)) - \nu m(\cdot,w),
\end{equation}
which may first be understood in the sense of distributions.  The parameter $\nu\ge 0$ is the extinction rate of the firms
and the source term $\eta $ stands for the exogenous creation of firms.
Note that the latter term depends on the value  $u$. We make the following assumption:
\begin{assumption}[Assumptions on $\nu$ and $\eta$] \label{ass:secFP:1}
We assume that $\nu$ is positive ($\nu>0$),  that $\eta$ is a continuous function on $[0,+\infty)\times \R$, and that there exists a continuous probability density  $\hat \eta:\R_+\rightarrow \R_+$ with a compact support contained in $(0,+\infty)$ 
and a positive constant $\hat c\ge 1$ such that  for all $k\ge 0$ and $v\in \R$, 
\begin{displaymath}
  \frac 1  {\hat c} \hat  \eta(k) \le \eta( k,v) \le   \hat c \hat  \eta(k).
\end{displaymath}
\end{assumption}
Equation \eqref{eq:chap:MFG_model:FP} is supplemented with the condition
\begin{equation}\label{eq:12}
   \nu \int_{\R_+}  m(k,w) dk= \int_{\R_+} \eta(k, u(k,w))  dk.
\end{equation}
Since $H_q(k, \frac {\partial u}{\partial k} (k,w),w)>0$ for small values of $k$ and $H_q(k, \frac {\partial u}{\partial k} (k,w),w)<0$ for large values of $k$,  see Theorem \ref{th:secHJ:main}, \eqref{eq:12} is a weak way to say that the flux \\
$m(\cdot,w)H_q(\cdot, \frac {\partial u}{\partial k} (\cdot,w),w)$ vanishes at $k=0$ and as $k\to +\infty$.

\begin{proposition}\label{prop:FP_1}
  Under  Assumptions \ref{ass:secHJ:1}, \ref{ass:secHJ:3} and  \ref{ass:secFP:1}, 
 the unique solution of \eqref{eq:chap:MFG_model:FP}-\eqref{eq:12} is given by 
\begin{equation}
  \label{eq:chap:MFG_model:FP10}
  \begin{split}
     m(k,w)=\\
    \left\{ \begin{array}[c]{rcl}
              \ds \frac{1}{	b(k,w) }    \int_0^{k}\eta(\kappa, u(\kappa,w))
              \exp\left(-\ds \int_\kappa^k\frac{\nu}{    	b(z,w)}dz\right)d\kappa , \quad&\text{ if } &\quad k <\kappa^*(w),\\
              \ds  -\frac{1}{	b(k,w)}\int_k^{+\infty}\eta(\kappa, u(\kappa,w))\exp
              \left(\ds \int_k^\kappa\frac{\nu}{ b(z,w)}dz\right)d\kappa, \quad  &\text{ if } & \quad   k>\kappa ^*(w),
            \end{array}\right.
        \end{split}
  \end{equation}
 where, for brevity, $b(k,w)$ stands for the optimal investment when the capital is $k$:
 \begin{equation}
   \label{eq:b}
	b(k,w)=H_q\left(k,\frac {\partial u} {\partial k}(k,w),w\right).   
 \end{equation}
\end{proposition}
A key step in  the proof of Proposition \ref{prop:FP_1} consists of showing that the quantities in the right hand side of    \eqref{eq:chap:MFG_model:FP10} are well defined. This comes from an intermediate result which states that, under Assumptions \ref{ass:secHJ:1} and \ref{ass:secHJ:3}, 
$\left|H_q(k,\frac {\partial u} {\partial k}(k,w),w) \right|= O(| k-\kappa^*(w)|)$ for  $k$ in a neighborhood of $\kappa^* (w)$.
The latter information implies that, with the optimal investment strategy, a firm starting with a capital $k_0\not= \kappa^* (w)$  never reaches  $\kappa^* (w)$,
 even though its capital does tend to $\kappa^* (w)$ as $t\to \infty$.

\begin{remark} \label{sec:rk_support}
  Note that  \eqref{eq:chap:MFG_model:FP10}  implies that  $   \frac 1{\nu \hat c} \le   \int_{\R_+} m(k,w) dk\le \frac {\hat c} \nu$ (the two bounds do not depend on $w$).
Moreover, the support of $m(\cdot, w)$  is contained in the interval
\begin{displaymath}
  \left [ \min\left( \min\{k\in \hbox{support}(\hat \eta)\}, \kappa^*(w)\right),
    \max\left( \max\{k\in \hbox{support}(\hat \eta)\}, \kappa^*(w)\right)  \right ].
\end{displaymath}
Hence, from the continuity of $w\mapsto \kappa^*(w)$, for any compact set $K\subset \R_+^d$, there exists
a compact interval of  $\R_+$  containing the supports of  $m(\cdot, w)$  for all $w\in K$.
\end{remark}

\subsection{Equilibria}
\label{sec:equilibria}
 The supply of inputs is assumed to be of the form $S(w)$, where $w\in \R_+^d$  is the collection of prices.\\
At the equilibrium, we require that the  clearing condition on the markets of inputs holds, i.e.
\begin{equation}
  \label{eq:clearing_condition}
  S(w)= \int_{\R_+}  \ell^* \left(k, w\right)   m(k,w) dk.
\end{equation}
where $ \ell_i^* \left(k, w\right) =   - \frac{\partial f }{\partial w_i} (k,w)$, and $m(\cdot, w)$ is the solution of   \eqref{eq:chap:MFG_model:FP}-\eqref{eq:12}.\\
We  aim at proving the existence of equilibria by using Brouwer degree theory.
This requires additional assumptions:
 \begin{assumption}[Assumptions on the supply]
   \label{ass:chap:MFG_model:S}
   The function  $S: \;\R_+^d\rightarrow \R_+^d$ is of the form $S(w)= D_w\Phi(w)$,
   where 
  \begin{enumerate}
  \item $\Phi:\; \R_+^d\rightarrow \R $ is $C^1$ regular and strictly convex
  \item $\Phi$ is bounded from below  (for example by $0$)
  \item $\Phi$ is coercive, i.e.  $\lim_{\|w\|\to \infty} \Phi(w)=+\infty$.
  \end{enumerate}
 \end{assumption}
  \begin{remark}
    \label{ass:chap:MFG_model:Sr1}
    The Legendre-Fenchel transform of $\Phi$,  $\Phi^*(S)= \sup_{w\in \R_+^d} S\cdot w- \Phi(w)$  is  convex and semi-continuous  on $ \R_+^d$ with values in $(-\infty,+\infty]$. It can be interpreted as a collective cost or disutility associated to  the supply of inputs. Concerning raw materials, it may be linked to their scarsity or to the environmental/social damages caused by their production. For manpower, the disutility captures negative  effects of labour on the welfare  of the workers. 
  \end{remark}

  \paragraph{Examples}
    \begin{enumerate}
    \item If, for any $i=1,\dots, d$, the $i$-th component of the supply function is a non negative, continuous and increasing function of $w_i$, i.e. $S_i(w)=S_i(w_i)$, then  Assumption \ref{ass:chap:MFG_model:S} is satisfied with $\Phi(w)= \sum_{i=1}^d \int_0 ^{w_i}  S_i(t) dt$.
    \item Given two positive numbers $\sigma$ and $w_0$, if
      \begin{displaymath}
        S_i(w)= \frac { \exp( w_i/\sigma)} {\sum_{j=0} ^d \exp( w_j/\sigma)}
      \end{displaymath}
       for all $i=1,\dots, d$,\\
      then  Assumption \ref{ass:chap:MFG_model:S} is satisfied with $\Phi(w)= \sigma \log\left(\sum_{j=0}^d \exp( w_j/\sigma)\right)$. In the limit $\sigma\to 0$, the price $w_0$ can be seen as a reserve price under which the production factors cannot be acquired.
    \end{enumerate}
  
 Set
     \begin{equation}\label{eq:supply0} 
      g(k,w)=   f(k,w) +\delta k=\sup_{\ell\in \R_+^d} F(k,\ell) -w\cdot \ell, 
     \end{equation}
 which can be seen as the Legendre-Fenchel transform of $\ell\mapsto -F(k,\ell)$ evaluated at $-w$. From Assumption   \ref{ass:secHJ:3}, 
 $g(k,w)$ is finite, nonnegative and achieved by the unique maximizer
 $\ell^* (k,w)\in \R_+^d$, and   $\ell^*$  is $C^1$ on $(0,+\infty)\times (0,+\infty)^d$.

    A further technical assumption involving both $F$ and the fixed measure $\hat \eta$ arising in  Assumption  \ref{ass:secFP:1} will be needed:
 \begin{assumption}
   \label{ass:chap:MFG_model:technical_assump}
   Let $\one\in \R^d$ be defined by   $\one=(1,\dots,1)$. We assume that there exists $\epsilon\in (0,1)$ such that for all $\lambda\in [0,1]$, if
   \begin{equation}\label{eq:10004}
     \begin{split}
&     \Phi(w)+\int_0^{+\infty} g(k,w) \Bigl( (1-\lambda) d\hat \eta (k)  +   \lambda dm(k,w)\Bigr) \\ \le    & \Phi(\one)+\int_0^{+\infty} g(k,\one) \Bigl( (1-\lambda) d\hat \eta (k)  +   \lambda dm(k,w)\Bigr)       
     \end{split}
   \end{equation}
   then
   \begin{equation}
     \label{eq:10012}
     w\in\left(\epsilon, \frac 1 \epsilon\right)^d.
   \end{equation}

 \end{assumption}
 \begin{remark}
   The proof that  Assumption \ref{ass:chap:MFG_model:technical_assump} holds for classical examples of production functions will be given in Section  \ref{sec:equilibrium}.
 \end{remark}
 
 Section \ref{sec:equilibrium} will be devoted to the proof of the following existence result:
 \begin{theorem}[Existence of equilibria]\label{th_ex_equil}
   Under Assumptions  \ref{ass:secHJ:1}, \ref{ass:secHJ:3},  \ref{ass:secFP:1}, \ref{ass:chap:MFG_model:S} and  \ref{ass:chap:MFG_model:technical_assump},
there exists an equilibrium, i.e. $w\in (0,+\infty)^d$ such that
 the market clearing condition \eqref{eq:clearing_condition} holds with $m(\cdot,w)$ and $u(\cdot,w)$ uniquely defined respectively by Proposition \ref{prop:FP_1} and Theorem \ref{th:secHJ:main}.
\end{theorem}

\subsection{Classical examples of utility and production functions}
\label{sec:import-exampl-util}
 \subsubsection{Examples of utility functions}
  The  constant relative risk aversion (CRRA) utility  is a common example of  a utility that satisfies Assumption  \ref{ass:secHJ:1}: 
   \begin{displaymath}
     U(c) = \ln(c)\quad    \hbox{ or } \quad  U(c) = \frac 1 b c^b \quad \hbox { with }b\in(0,1)
   \end{displaymath}

\medskip

\subsubsection{Examples of production functions}
\begin{enumerate}
\item A classical example is the Cobb-Douglas  function:
  \begin{displaymath}
F(k,\ell)= A k^\alpha \ell^\beta,    
  \end{displaymath}
 where  $\beta \in (0,1)^d$, $\sum_{i=1}^d \beta_i <1$, 
  $\ell^\beta=\prod_{i=1}^d \ell_i^{\beta_i}$,   and  $0< \alpha < 1-\sum_{i=1}^d \beta_i$.
Let us set $|\beta|=\sum_{i=1}^d \beta_i$.
   In this example, the parameters $\beta$ and $\alpha$  respectively stand for the  elasticities of the output 
with respect to the different inputs  and to the capital, and $A>0$ is a global factor of productivity.
The net output is given by
\begin{equation}
  \label{eq:chap:examples:net_output_CobbDouglas}
  f(k,w) = (1-|\beta|) \left(Ak^\alpha\prod_{i = 1}^d\left(\frac{\beta_i}{w_i}\right)^{\beta_i}\right)^\frac{1}{1-|\beta|}-\delta k.
\end{equation}
It can be checked that the first order partial derivatives of $f$ with respect to $k$ and $w_i$ are  
\begin{equation}
  \label{eq:dfdk_cobb_douglas}
  \frac{\partial f}{\partial k}(k,w) = \alpha \left(A\prod_{j = 1}^d\left(\frac{\beta_j}{w_j}\right)^{\beta_j}\right)^\frac{1}{1-|\beta|}
k^{-\frac{1-\alpha-|\beta|}{1-|\beta|}}-\delta,
\end{equation}
and
\begin{equation}
  \label{eq:dfdwcobbdouglas}
  	\frac{\partial f}{\partial w_i}(k,w) = - \left(Ak^\alpha\prod_{j = 1}^d\left(\frac{\beta_j}{w_j}\right)^{\beta_j}\right)^\frac{1}{1-|\beta|}\frac{\beta_i}{w_i}\le 0.
      \end{equation}
It is easy to see that Assumption \ref{ass:secHJ:3} is satisfied. In particular, \\$\lim_{k\to +\infty }  \frac{\partial f}{\partial k}(k,w) =-\delta$.
The capital $\kappa^*(w)$ in  \eqref{eq:1} is given by
\begin{equation}	
  \label{eq:chap:examples:sec:Cobb_Douglas:target_capital}
  \kappa^*(w) = \left(\frac{\alpha}{\alpha+\rho}\right)^\frac{1-|\beta|}{1-\alpha-|\beta|}
  \left(A\prod_{j = 1}^d\left(\frac{\beta_j}{w_j}\right)^{\beta_j}\right)^\frac{1}{1-\alpha-|\beta|}.
\end{equation}

\item We now consider a   production  function with a constant elasticity of substitution:
  \begin{displaymath}
 F(k,\ell) = \left(k^\alpha + \sum_{i = 1}^d\ell_i^{\beta_i}\right)^\gamma,   
  \end{displaymath}
where   $\alpha \in(0,1)$, $\beta\in(0,1)^d$ and $\gamma \in (0,1)$.
For any $(k,w)\in \R_+\times(0,+\infty)^d$, it can be checked that  there exists a unique parameter  $\lambda(k,w)>0$ such that 
\begin{equation}
  \label{eq:lambda_CES}
  \lambda\left(k^\alpha + \sum_{j = 1}^d\left(\frac{\lambda\beta_j}{w_j}\right)^\frac{\beta_j}{1-\beta_j}\right)^{1-\gamma}= \gamma.
\end{equation}
The net output is then
\begin{displaymath}
  f(k,w) = \left(k^\alpha+\sum_{j=1}^d\left(\frac{\lambda(k,w)\beta_j}{w_j}\right)^\frac{\beta_j}{1-\beta_j}\right)^\gamma - \sum_{j = 1}^dw_i\left(\frac{\lambda(k,w)\beta_j}{w_j}\right)^\frac{1}{1-\beta_j}-\delta k.
\end{displaymath}
It can be checked that the first order partial derivatives of $f$ with respect to $k$ and $w_i$ are  
\begin{equation}
  \label{eq:df/dk_constant_elasticity}
  \frac{\partial f}{\partial k}(k,w) = \alpha \lambda(k,w)k^{\alpha - 1}- \delta, 
\end{equation}
and
\begin{equation}
		\label{eq:Dwf_CES}
		\frac{\partial f}{\partial w_i}(k,w) = - \left(\frac{\lambda(k,w)\beta_i}{w_i}\right)^\frac{1}{1-\beta_i}<0.
\end{equation}
Assumption \ref{ass:secHJ:3} is satisfied. In particular, $\lim_{k\to +\infty }  \frac{\partial f}{\partial k}(k,w) =-\delta$.
The capital $\kappa^*(w)$ in  \eqref{eq:1} is the unique solution of
\begin{equation}
  \label{eq:target_capital_CES}
  \alpha\lambda(\kappa^*(w),w) (\kappa^*(w))^{\alpha-1} = \delta+\rho.
\end{equation}
\end{enumerate}

\section{The optimal control problem of a single firm}
\label{sec:optim-contr-probl-1}
In this section, we assume that $w$,  the prices of the production factors, is a fixed vector in $(0,+\infty)^d$. Thus, in order to alleviate the notation,
we everywhere omit the dependency upon $w$; for example we write $H(k,q)$ and $u(k)$ instead of $H(k,q,w)$ and $u(k,w)$. Similarly, we set 
$u'(k)= \frac {\partial u}{\partial k} (k,w)$ and $f'(k)= \frac {\partial f}{\partial k} (k,w)$.

\bigskip

The proof of Theorem \ref{th:secHJ:main} is simpler when  $\delta=0$
because $f$ is positive on $(0,+\infty)$. We will first focus on the latter case, then we will address the other case, i.e. $\delta >0$.

\subsection{The particular case where  $\delta=0$ }

\subsubsection{Some properties of the Hamiltonian}
\label{sec:secHJ:propertiesH}

\begin{lemma}\label{lem:secHJ:Hconvex}
  Under Assumption  \ref{ass:secHJ:1},  for any  $k>0$, the function $q \mapsto H\left(k,q\right)$, defined on $(0,+\infty)$,
  is strictly convex and of class $C^2$.
\end{lemma}
\begin{proof}
From Assumption~\ref{ass:secHJ:1}, the function $U'$ is one to one on $(0,+\infty)$. 
Let $c^*$ denote the inverse function, which is decreasing  and $C^1$ on $(0,+\infty)$; its derivative is $q\mapsto 1/ U''(c^* (q))$.
For any $q>0$, $c^*(q)>0$ is the  unique consumption which achieves the  supremum in (\ref{eq:chap:MFG_model:secGeneralities:32}),
because $U'(c^*(q)) = q$. The derivative of  $q\mapsto H(k,q)$ is 
\begin{equation}
   \label{eq:chap:MFG_model:secHJ:8}
  H_q(k,q)=  -c^*(q)+f(k).
\end{equation}
Hence,  $q\mapsto   H(k,q)$ is $C^2$  on $(0,+\infty)$ and    $H_{qq}(k,q)=  -1/{U''( c^* (q))}>0$. This implies the strict convexity of  $q\mapsto   H(k,q)$.
\end{proof}
\begin{remark}\label{sec:some-prop-hamilt}
  Note that the consumption achieving the supremum in  (\ref{eq:chap:MFG_model:secGeneralities:32}) does not depend on $k$.
\end{remark}
\begin{lemma}\label{lem:secHJ:Hmin}
  We make Assumptions \ref{ass:secHJ:1} and \ref{ass:secHJ:3} and suppose furthermore that  $\delta =0$, hence $\lim_{k\to +\infty}  f'(k)= 0$.
Then, for any  $k>0$,
  \begin{eqnarray}
    \label{eq:chap:MFG_model:secHJ:5}
   \ds  \min_{q>0}H(k,q) &=& U(f(k)),\\
     \label{eq:chap:MFG_model:secHJ:6}
\ds    \argmin_{q>0}H(k,q) &=& \ds \left\{U'(f(k))\right\}.
  \end{eqnarray}
\end{lemma}
\begin{proof}
For $k>0$, $f(k)>0$ by Remark \ref{rem:1}. From (\ref{eq:chap:MFG_model:secHJ:8}), $H_q(k,q)=0$ if and only if $c^*(q)=f(k)$, i.e.  $q=U'(f(k))$.
This proves that the infimum of the strictly convex function $q\mapsto H(k,q)$ is a minimum,  which is achieved by  $q=U'(f(k))$.  
The minimal value is  $U(c^* (U'(f(k))))=U(f(k))$.
\end{proof}
\begin{remark}\label{rem:2-2}
 From Assumption \ref{ass:secHJ:1}, we see that if $f(0)=0$, then \\ $\lim_{k\to 0}  U'(f(k)) = +\infty$.
  On the  contrary, from Remark \ref{rem:1}, if $f(0)>0$, then $U'\circ f$ remains bounded on bounded subsets of $[0,+\infty)$.
\end{remark}
\begin{lemma}\label{sec:some-prop-hamilt-1}
  Under Assumption \ref{ass:secHJ:1},
  \begin{eqnarray}
      \label{eq:chap:MFG_model:secHJ:23}
  \ds  \lim_{q\to 0+} H(k,q)&=& \ds \lim_{c\to +\infty} U(c)-cU'(c) =  \lim_{c\to +\infty} U(c) \in (-\infty,+\infty],\\
    \label{eq:chap:MFG_model:secHJ:24}
  \ds  \lim_{q\to 0+} H_q(k,q)&=&-\infty.
  \end{eqnarray}
  \end{lemma}
\begin{proof}
Since $c^*$ is the inverse of $U'$ on $(0,+\infty)$, Assumption \ref{ass:secHJ:1} implies that $\lim_{q\to 0} c^*(q)=+\infty$.
 Therefore, from   \eqref{eq:chap:MFG_model:secHJ:8}, $\lim_{q\to 0} H_q (k,q)=-\infty$.

  We know that $U$ is increasing: let us set $\ell_1= \lim_{c\to +\infty} U(c)=\sup_{c\ge  0} U(c) \in (-\infty, +\infty]$.
  On the other hand, the function $c\mapsto U(c)-cU'(c)$ is increasing in $\R_+$, because its derivative is $c\mapsto -cU''(c)$; let us set $\ell_2= \lim_{c\to +\infty} U(c)-cU'(c) \in (-\infty,+\infty]$.
  
 Since $H(k,q)\sim U(c^* (q))- c^*(q) U'(c^*(q))$ as $q\to 0$, we see that  $\ds \lim_{q\to 0} H(k,q)= \ell_2$.
   
  We need to compare $\ell_1$ and $\ell_2$. It is obvious that $\ell_2\le \ell_1$. We wish to prove that  $\ell_2= \ell_1$.
  We argue  by contradiction and  assume that $\ell_2<\ell_1$. We make out two cases:
  \begin{enumerate}
  \item   $\ell_1\in \R$ and $\ell_2<\ell_1$: we see that $c U'(c)$ tends to $\ell_1-\ell_2>0 $ as $c$ tends to $+\infty$. This implies that $U(c)$ blows up like a logarithm of $c$ as $c$ tends to $+\infty$, in contradiction with the fact that $\ell_1<+\infty$. Therefore, if $\ell_1$ is finite, then $\ell_1=\ell_2$.
  \item  $\ell_1=+\infty$ and $\ell_2\in \R$. 
We see that  $c U'(c)=   U(c)-\ell_2  + o(1)$ where $\lim _{c\to \infty} o(1)=0$.
Using Gronwall lemma, we deduce that there 
exists a real number $\chi$ such that $U(c)= \chi c +\ell_2 + o(1)$. Since $U(c)\to +\infty$ as $c\to +\infty$, we see that $\chi>0$. We deduce that $\lim _{c\to \infty} U'(c)=+\infty$ in contradiction with Assumption \ref{ass:secHJ:1}.
  \end{enumerate}
  The proof is complete.
\end{proof}

Lemmas \ref{lem:secHJ:Hconvex} and \ref{lem:secHJ:Hmin} above allow us to define the increasing and decreasing parts of the Hamiltonian:
\begin{definition}\label{def:monotone_enveloppes}
  We make Assumptions \ref{ass:secHJ:1} and \ref{ass:secHJ:3} and suppose furthermore that
  $\delta=0$. 
  \begin{itemize}
  \item Define the sets 
    \begin{eqnarray*}
      \Theta^\uparrow&=&\left\{(k,q)\hbox{ such that } k>0 \hbox{ and } q\geq U'(f(k))\right\},\\
      \Theta^\downarrow&=&\left\{(k,q)\hbox{ such that } k>0 \hbox{ and } q\leq U'(f(k))\right\}.
       \end{eqnarray*}
     \item Let $H^\uparrow(\cdot,\cdot)$ be the restriction of $H(\cdot,\cdot)$ to $ \Theta^\uparrow$.  The function $q\mapsto H^\uparrow(k,q)$ is increasing in $[U'(f(k)), \infty)$.
     \item Let $H^\downarrow(\cdot,\cdot)$ be the restriction of $H(\cdot,\cdot)$ to $ \Theta^\downarrow$.  The function $q\mapsto H^\downarrow(k,q)$ is decreasing on $(0,U'(f(k))]$.
  \end{itemize}
The graphs of $H(k,\cdot)$, $H^\uparrow (k,\cdot)$ and $H^\downarrow (k,\cdot)$ are displayed on Figure \ref{fig:1}.
\end{definition}

\begin{center}
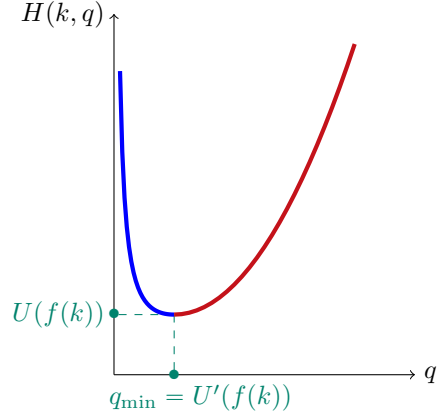

  \begin{tikzpicture}[scale=0.8]  
    \draw[->] (0,0) -- (5,0) node[right] {$q$};
    \draw[->] (0,0) -- (0,6)    ;
    \draw[color=PineGreen, dashed] (1,1) -- (1,0)    ;
     \draw[color=PineGreen, dashed] (1,1) -- (0,1)    ;
    \draw[domain=0.1:1, ultra thick, variable=\x, color=blue] plot ({\x}, {1+  0.5*(\x+1/\x-2)});
 
\draw[domain=1:4, ultra thick, variable=\x, color=Red] plot ({\x}, {1+0.5*(\x-1)*(\x-1)});
  
    \draw(0,6)      node[left] {$ H(k,q)$};
    \draw[color=PineGreen] (1,0) node{$\bullet$};
    \draw(1,0)   node[below, color=PineGreen] {$\quad \quad  q_{\min}= U'(f(k))$};
    \draw[color=PineGreen] (0,1) node{$\bullet$};
    \draw(0,1)   node[left, color=PineGreen] {$ U(f(k))$};
  \end{tikzpicture}
  \captionof{figure}{\label{fig:1} The bold line (blue and red) is the graph of the function $H(k,\cdot)$. The blue  line is the graph of $H^{\downarrow}(k,\cdot)$.    The  red line is the graph of $H^{\uparrow}(k,\cdot)$. In the present figure,  $\lim_{q\to 0_+} H(k,q)=+\infty$, but it is also possible that  $\lim_{q\to 0_+} H(k,q)\in \R$.}
\end{center}

\begin{lemma}\label{sec:secHJ:H_C1_reg}
Under the same assumptions as in Lemma~\ref{lem:secHJ:Hmin},
 $H^\downarrow(\cdot,\cdot)$  (respectively $H^\uparrow(\cdot,\cdot)$)   is of class $C^1$ on $\Theta^\downarrow$ (respectively  $\Theta^\uparrow$).
\end{lemma}
\begin{proof}
  We have already seen in the proof of Lemma \ref{lem:secHJ:Hconvex} that  $q\mapsto H(k,q)$ is of class $C^2$.
  Moreover, from  Assumption \ref{ass:secHJ:3}, $k\mapsto f(k)q$ is of class $C^1$, so  $k\mapsto H(k,q)$ is also of class $C^1$. Hence $(k,q) \mapsto H^\downarrow(k,q) $ is of class $C^1 $ on  $\Theta^\downarrow$, and so is $(k,q) \mapsto H^\uparrow(k,q) $ on  $\Theta^\uparrow$.
\end{proof}
\subsubsection{General orientation}\label{sec:general-orientation}
 Heuristically, if $u$ is a classical solution of (\ref{eq:chap:MFG_model:secGeneralities:HJ}) such that $u'(k)>0$ for $k>0$ and
 $u''$ is locally bounded, 
then, taking the derivative of  (\ref{eq:chap:MFG_model:secGeneralities:HJ}), we get that for $k>0$,
\begin{displaymath}
  \left(f'(k)-\rho\right) u'(k) = -H_q\left(k,u'(k)\right)  u''(k) .
\end{displaymath}
We deduce that if the optimal investment is $0$, i.e.
 $H_q\left(k,u'(k)\right)=0$,
then 
\begin{equation}
  \label{eq:2}
f'(k)=\rho.
\end{equation}
From Assumption \ref{ass:secHJ:3}, (\ref{eq:2}) has a unique solution which we name $\kappa^*$ (note that $\kappa^*$ depends on $w$, see (\ref{eq:1}) in Theorem~\ref{th:secHJ:main}).
\\
Moreover,  $H_q\left(\kappa^*,u'(\kappa^*)\right)=0$ implies that
 $  u'(\kappa^*)= U'(f(\kappa^*))$ and $H(\kappa^*,u'(\kappa^*))=U(f(\kappa^*))$, 
see Figure \ref{fig:1}. Hence, from (\ref{eq:chap:MFG_model:secGeneralities:HJ}), we deduce that
$u(\kappa^*)=U(f(\kappa^*))/\rho$.
\\
On the other hand, because of the state constraint,  we expect that 
$H_q\left(k,u'(k)\right)$ is  positive for small values of $k$. Hence, we expect that for a classical state constrained solution $u$ of 
 (\ref{eq:chap:MFG_model:secGeneralities:HJ}),
\begin{displaymath}
H\left(k, u'(k)\right)= \left\{ \begin{array}[c]{ll}
    H^{\uparrow} \left(k, u'(k)\right), \quad \hbox{if } k<\kappa^*,\\
    H^{\downarrow} \left(k, u'(k)\right), \quad \hbox{if } k>\kappa^*.
  \end{array}\right.
\end{displaymath}
Therefore, we are going to look for $u$ as the solution of  two ordinary differential equations in $(0,\kappa^*)$ and $(\kappa^*, +\infty)$ which respectively  involve 
the inverse functions  of $q\mapsto  H^{\uparrow} (k,q)$ and $q\mapsto  H^{\downarrow}(k,q)$, with the boundary condition
\begin{displaymath}
  u(\kappa^*)=U(f(\kappa^*))/\rho.
\end{displaymath}
In order to carry out this program, we need to consider the inverse functions of
$q\mapsto H^{\uparrow}(k,q) $ and  $q\mapsto H^{\downarrow}(k,q)$:
\begin{definition}
   We make Assumptions \ref{ass:secHJ:1} and \ref{ass:secHJ:3} and suppose furthermore that
  $\delta=0$.
  \begin{itemize}
  \item Define the sets 
    \begin{eqnarray}
       \label{eq:chap:MFG_model:secHJ:12}
      \Omega^\uparrow&=&\left\{(k,v):\; k\in\left(0,\kappa^*\right] \hbox{ and } \rho v\in\left(U(f(k)),+\infty\right)\right\},\\
       \label{eq:chap:MFG_model:secHJ:13}
      \Omega^\downarrow&=&\left\{(k,v):\ k\in\left[\kappa^*,+\infty\right) \hbox{ and } \rho v\in\left(U(f(k)), \lim_{q\to 0 ^ +}  H(k,q)\right)\right\}.
       \end{eqnarray}
     \item Set
       \begin{eqnarray}
         \label{eq:chap:MFG_model:secHJ:10}
         \cF^\uparrow(k,v) &=& \left(H^\uparrow(k,\cdot)\right)^{-1}(\rho v),\quad\quad \hbox{for } (k,v) \in   \Omega^\uparrow,\\
         \label{eq:chap:MFG_model:secHJ:11}
         \cF^\downarrow(k,v) &=& \left(H^\downarrow(k,\cdot)\right)^{-1}(\rho v),\quad\quad \hbox{for } (k,v) \in   \Omega^\downarrow.
       \end{eqnarray}
       \end{itemize}
     \end{definition}

\bigskip

\paragraph{Program} Our program will be as follows: 
\begin{enumerate}
\item
Prove that the following Cauchy problem has a unique solution \\ $u^\downarrow: [\kappa^*,+\infty) \to \R$:
\begin{eqnarray}\label{eq:chap:MFG_model:secHJ:ODEdown1}
 \frac  {du^\downarrow}{dk}(k) &=& \cF^\downarrow (k,u^\downarrow(k)),\quad \quad\quad \hbox{for } k\ge  \kappa^*,\\
\label{eq:chap:MFG_model:secHJ:ODEdown2}
(k,u^\downarrow(k))&\in& \Omega^\downarrow,\quad \quad\quad\quad\quad \quad\quad \hbox{for } k> \kappa^*,\\
\label{eq:chap:MFG_model:secHJ:ODEdown3} 
  u^\downarrow(\kappa^*) &=& \frac{1}{\rho}U(f(\kappa^*)).
\end{eqnarray}
\item Prove that the following   Cauchy problem has a unique solution $u^\uparrow: (0,\kappa^*] \to \R$:
\begin{eqnarray}\label{eq:chap:MFG_model:secHJ:ODEup1}
  \frac {d u^\uparrow}{dk}(k) &=& \cF^\uparrow (k,u^\uparrow(k)),\quad \quad\quad \hbox{for } k\le  \kappa^*,\\
\label{eq:chap:MFG_model:secHJ:ODEup2}
(k,u^\uparrow(k))&\in& \Omega^\uparrow,\quad\quad\quad \quad\quad\quad\quad \hbox{for } 0<k< \kappa^*,\\
\label{eq:chap:MFG_model:secHJ:ODEup3} 
  u^\uparrow(\kappa^*) &=& \frac{1}{\rho}U(f(\kappa^*)).
\end{eqnarray}
\item Prove that the function $u$ which coincides with $  u^\uparrow$ on $[0,\kappa^*]$ and  $  u^\downarrow$ on $[\kappa^*,+\infty)$
is the  solution of (\ref{eq:chap:MFG_model:secGeneralities:HJ})-(\ref{eq:chap:MFG_model:secHJ:3}). 
\end{enumerate}

\bigskip

Before starting this program, let us state a useful lemma:
\begin{lemma}\label{sec:secHJ:F_C1_reg}
 Under the same assumptions as in Lemma~\ref{lem:secHJ:Hmin}, $\cF^\downarrow(\cdot,\cdot)$  (respectively $\cF^\uparrow(\cdot,\cdot)$)   is of class $C^1$
  on $\Omega ^\downarrow$ (respectively  $\Omega^\uparrow$).
\end{lemma}

\begin{proof}
We skip the proof for brevity and refer to \cite{Petit2022phd}, which contains an extended version of the present paper. 

\end{proof}

\subsubsection{The Cauchy problem (\ref{eq:chap:MFG_model:secHJ:ODEdown1})-(\ref{eq:chap:MFG_model:secHJ:ODEdown3})}
\label{sec:cauchy-probl-refeq:c}
Let us first consider the maximal  solution  $\phi_\lambda$ of the following Cauchy problem:
\begin{eqnarray}\label{eq:chap:MFG_model:secHJ:14}
  \phi_\lambda'(k) &=& \cF^\downarrow (k,\phi_\lambda(k)),\quad \quad\quad \hbox{for } k\ge \kappa^*,\\
\label{eq:chap:MFG_model:secHJ:15}
(k,\phi_\lambda(k))&\in& \Omega^\downarrow,\\
  \label{eq:chap:MFG_model:secHJ:16}
  \phi_\lambda(\kappa^*) &=& \lambda,
\end{eqnarray}
for $\lambda$ such that $(\kappa^*,\lambda)\in\Omega^\downarrow$, see \eqref{eq:chap:MFG_model:secHJ:13}. Cauchy-Lipschitz theorem may be applied because 
  $\cF^\downarrow$ is regular enough on $\Omega^\downarrow$.
  After having proved the existence and uniqueness of $\phi_\lambda$,
we will let  $\lambda$ tend to $U(f(\kappa^*))/\rho$ and obtain that the sequence $\phi_\lambda$ converges to a solution of (\ref{eq:chap:MFG_model:secHJ:ODEdown1})-(\ref{eq:chap:MFG_model:secHJ:ODEdown3}). One reason for not applying directly the standard existence results
 to the Cauchy problem with initial condition $\lambda= U(f(\kappa^*))/\rho$
is that $\cF^\downarrow(\cdot,\cdot)$ is not regular at the boundary of $\Omega^\downarrow$. In particular, $ v\mapsto \cF^\downarrow (\kappa^* ,v)$ is not Lipschitz continuous in the neighborhood of $(\kappa^*, U(f(\kappa^*))/\rho)$.
Moreover, the point $(\kappa^*,U(f(\kappa^*))/\rho)$ belongs to the boundary of $\Omega^\downarrow$;
this forbids the direct use of Cauchy-Peano-Arzel{\`a} theorem for obtaining the existence of a solution.

\begin{proposition}\label{prop:secHJ:vlambda}
  We make  Assumptions \ref{ass:secHJ:1} and \ref{ass:secHJ:3} and suppose furthermore that  $\delta=0$. 
  For every $\lambda$ such that $(\kappa^*,\lambda)\in\Omega^\downarrow$,   there exists a unique global solution $ \phi_\lambda$  of \eqref{eq:chap:MFG_model:secHJ:14}-\eqref{eq:chap:MFG_model:secHJ:16} in $[\kappa^*, +\infty)$. The function  $ \phi_\lambda$ is increasing and strictly concave. 
\end{proposition}
\begin{proof}
Setting $\Theta(k)= (k, \phi_\lambda(k))$,  it is convenient to rewrite \eqref{eq:chap:MFG_model:secHJ:14}-\eqref{eq:chap:MFG_model:secHJ:16} in the equivalent form:
  find $k\mapsto \Theta(k) \in \Omega^\downarrow$ such that
  \begin{eqnarray}
    \label{eq:chap:MFG_model:secHJ:17}
    \Theta'(k)&=& \left (1,  \cF^\downarrow (\Theta(k))\right) ,\quad  k\ge \kappa^*,\\
    \label{eq:chap:MFG_model:secHJ:18}
    \Theta(\kappa^*) &=&(\kappa^*, \lambda).
  \end{eqnarray}
  We may apply Cauchy-Lipschitz theorem;  indeed, from Lemma, the map \\ $\Theta\mapsto  \left (1, \cF^\downarrow (\Theta)\right) $ is $C^1$ on $\Omega^\downarrow$. Therefore,  there exists a unique maximal solution $\Theta$ of \eqref{eq:chap:MFG_model:secHJ:17}-\eqref{eq:chap:MFG_model:secHJ:18} in $[\kappa^*, \bar k)$. We observe that for $k\in [\kappa^*,\overline{k})$, 
  $\phi_\lambda'(k)  = \cF^\downarrow(k,\phi_\lambda(k))>0$, so $\lim_{k\to \bar k^-} \phi_\lambda(k)$ exists.
  Moreover, by taking the derivative, 
  \begin{displaymath}
  \phi_\lambda''(k) =\frac {\rho-f'(k) }    { H_q\left(k,\phi_\lambda'(k) \right)} \phi_\lambda'(k)<0.      \end{displaymath}
Therefore $\phi_\lambda$ is strictly concave in $[\kappa^*, \bar k)$.

\medskip

If $\bar k<\infty$, then from Cauchy-Lipschitz theorem,
$\rho \lim_{k\to \bar k^-} \phi_\lambda(k)$ must be equal either to  $ U(f(\bar k))$ or to $\lim_{q\to 0} H(k,q)= \lim_{c\to +\infty} U(c)$
(which does not depend on $k$). Let us show by contradiction that both cases are impossible.
\begin{enumerate}
\item Assume first that $\rho \lim_{k\to \bar k^-} \phi_\lambda(k)=\lim_{q\to 0} H(k,q)= \lim_{c\to +\infty} U(c)$;
  let us make out two subcases:
  \begin{enumerate}
  \item If $ \lim_{c\to +\infty} U(c)=+\infty$,   then   $\lim_{k\to \bar k^-}   \phi_\lambda(k)   =+\infty$, which yields 
that $\lim_{k\to \bar k^-}  \cF^\downarrow(k,\phi_\lambda(k)) =0$. From (\ref{eq:chap:MFG_model:secHJ:17}), we see 
 that $\lim_{k\to \bar k^-} \phi_\lambda'(k) =0$, in contradiction with $\lim_{k\to  \bar k^-}  \phi_\lambda(k))   =+\infty$.
\item  If $ \lim_{c\to +\infty} U(c)= \ell\in \R$, then it is possible to extend continuously $\phi_\lambda$ to $\bar k$ by setting $\phi_\lambda(\bar k)= \ell/\rho$.
Since $H(k,0)=\ell$ for all $k$, we see that 
\begin{equation}
  \label{eq:3}
\cF^{\downarrow}(k,\ell/\rho)=0,\quad \hbox{ for all }k\ge \kappa^*. 
\end{equation}
 On the other hand, since $U'(c^* (q))=q$, Assumption \ref{ass:secHJ:1} implies that $\lim_{q\to 0} c^* (q)=+\infty$. This implies that
\begin{equation}
  \label{eq:4}
\frac {\partial \cF^\downarrow}{\partial v}(k,\ell/\rho)=0,\quad \hbox{ for all }k\ge \kappa^*. 
\end{equation}
  But (\ref{eq:3}) and (\ref{eq:4})   prevent the state $\ell/\rho$ to be reached in finite time by a solution of \eqref{eq:chap:MFG_model:secHJ:14}-\eqref{eq:chap:MFG_model:secHJ:15}; we have obtained the desired contradiction.
   \end{enumerate}
 \item Assume that  $\lim_{k\to \bar k^-} \phi_\lambda(k)=   U(f(\bar k))/ \rho$.    
It is then possible to extend continuously $\phi_\lambda$ to $\bar k$ by setting $\phi_\lambda(\bar k)=  U(f(\bar k))/ \rho$,
   and (\ref{eq:chap:MFG_model:secHJ:14}) holds in $[\kappa^*, \bar k]$.  On the other hand,
   \begin{equation}
     \label{eq:5}
\frac d {dk} \left( \frac {U(f(k))} \rho \right)- \cF^\downarrow \left(k, \frac {U(f(k))} \rho \right) = U'(f(k)) \frac{f'(k)-\rho}{\rho} <0,\quad \hbox{ for } k> \kappa^*,
   \end{equation}
 from the definition of $\kappa^*$ and Assumption \ref{ass:secHJ:3}. Thus, $k\mapsto U(f(k))/\rho$ is a subsolution of the ordinary differential equation
 satisfied by $\phi_\lambda$, which yields that $U(f(k))/\rho > \phi_\lambda(k)$ for $k< \bar k$. This is impossible, since 
$(k,  \phi_\lambda(k))\in \Omega^\downarrow$ for $k<\bar k$.
 \end{enumerate}
 We have proved that $\bar k=+\infty$. The unique maximal solution of  \eqref{eq:chap:MFG_model:secHJ:17}-\eqref{eq:chap:MFG_model:secHJ:18} is a global solution.
\end{proof}

Letting $\lambda$ tend to $U(f(\kappa^*))/\rho$, we shall prove the following result:
\begin{proposition}\label{prop:secHJ:ex_down}
  Under the same assumptions as in Proposition~\ref{lem:secHJ:philambda},
the  \\ Cauchy problem  (\ref{eq:chap:MFG_model:secHJ:ODEdown1})-(\ref{eq:chap:MFG_model:secHJ:ODEdown3}) has a unique solution 
 $u^\downarrow  \in C^1( [\kappa^*,+\infty))\cap C^2(\kappa^*,+\infty )$. Moreover $u^\downarrow$ is  strictly concave on $(\kappa^*, +\infty)$. 
\end{proposition}

\begin{proof}
  Consider a decreasing sequence $(\lambda_n)_{n\in\N}$, such for all $n\in \N$,
  $(\kappa^*,\lambda_n)\in\Omega^\downarrow$ and $\lim_{n\to \infty} \lambda_n=  U(f(\kappa^*))/\rho$.
  A direct consequence of Cauchy-Lipschitz theorem is that
  $\phi_{\lambda_n} (k)> \phi_{\lambda_{n+1}}(k)$ for all $k\ge \kappa^*$.
  On the other hand, we know that $\phi_{\lambda_n} (k)\ge U(f(k))/\rho$. This implies that there exists a function $\phi: [\kappa^*, +\infty)\to \R$ such that  $\phi_{\lambda_n}$ converge to
  $\phi$ pointwise as $n$ tends to $+\infty$.\\
 Since $(\phi_{\lambda_n})_{n\in\N}$ is a sequence of concave functions locally uniformly bounded,  we see from \cite[Theorem 3.3.3]{MR2041617} that the convergence is uniform on every compact set, so the limit $\phi$ is continuous.\\
  Since $\cF^\downarrow(\cdot,\cdot)$ is continuous on the closure of $\Omega^\downarrow$, 
 we may pass  to the  limit in the integral form of \eqref{eq:chap:MFG_model:secHJ:14}: for all $k\ge \kappa^*$,
  \begin{displaymath}
\phi(k) = \frac{1}{\rho}U(f(\kappa^*))+\int_{\kappa^*}^k\cF^\downarrow(\kappa,\phi(\kappa))d\kappa.    
  \end{displaymath}
This implies that $\phi\in C^1([\kappa^*, +\infty))$ and that $\phi$ satisfies \eqref{eq:chap:MFG_model:secHJ:ODEdown1} and  \eqref{eq:chap:MFG_model:secHJ:ODEdown3}. Hence $\phi$ is an increasing function.\\
On the other hand, \eqref{eq:5} implies that  $\phi(k)> U(f(k))/\rho$ for all $k>\kappa^*$. This shows that $\phi$ satisfies \eqref{eq:chap:MFG_model:secHJ:ODEdown2}.\\
Arguing as in the proof of Proposition \ref{prop:secHJ:vlambda}, we see that $\phi$ is $C^2$ on $(\kappa^*, +\infty)$ and strictly concave. 
We have proved  the existence of a solution of \eqref{eq:chap:MFG_model:secHJ:ODEdown1}-\eqref{eq:chap:MFG_model:secHJ:ODEdown3}.

Assume that there are two such solutions $\phi_1$ and $\phi_2$.
If there exists $k_0>\kappa^*$ such that $\phi_1(k_0)=\phi_2(k_0)$, then $\phi_1$ and $\phi_2$ coincide  from
Cauchy-Lipschitz theorem. Hence we may assume that  $\phi_1(k)<\phi_2(k)$ for $k>\kappa^*$.
Then, using the non increasing character of  $\cF^\downarrow(k,\cdot)$, we see that, for every $k>\kappa^*$,
\begin{displaymath}
0>\phi_1(k)-\phi_2(k) = \int_{\kappa^*}^k\cF^\downarrow(\kappa,\phi_1(\kappa))-\cF^\downarrow(\kappa,\phi_2(\kappa))d\kappa\geq 0.
\end{displaymath}
We have found a contradiction and achieved the proof of uniqueness.
\end{proof}

\subsubsection{The Cauchy problem \eqref{eq:chap:MFG_model:secHJ:ODEup1}-\eqref{eq:chap:MFG_model:secHJ:ODEup3}}
\label{sec:cauchy-probl-refeq:c2}
Also in this case,  $ \cF^\uparrow (k ,\cdot)$ is not Lipschitz continuous in the neighborhood of $(\kappa^*, U(f(\kappa^*))/\rho)$  
and $(\kappa^*, U(f(\kappa^*))/\rho)$ belongs to the boundary of $\Omega^\uparrow$. This prevents us
from applying directly  standard existence results to \eqref{eq:chap:MFG_model:secHJ:ODEup1}-\eqref{eq:chap:MFG_model:secHJ:ODEup3}.\\
For this reason, we start by considering the Cauchy problem:
\begin{eqnarray}\label{eq:chap:MFG_model:secHJ:19}
  \psi_{\epsilon,\lambda}'(k) &=& \cF^\uparrow (k,\psi_{\epsilon,\lambda}(k)),\quad 0<k \le \kappa^*,
\\
\label{eq:chap:MFG_model:secHJ:20}
(k,\psi_{\epsilon,\lambda}(k))&\in& \Omega^\uparrow,\\
  \label{eq:chap:MFG_model:secHJ:21}
  \psi_{\epsilon,\lambda}(\epsilon) &=& \lambda,
\end{eqnarray}
for $(\epsilon,\lambda)\in \Omega^\uparrow$, see  \eqref{eq:chap:MFG_model:secHJ:11} (thus $0<\epsilon<\kappa^*$).
As above, Cauchy-Lipschitz theorem may be applied to \eqref{eq:chap:MFG_model:secHJ:19}-\eqref{eq:chap:MFG_model:secHJ:21}. 
After having obtained the existence and uniqueness of a maximal solution $\psi_{\epsilon,\lambda}$, 
we will prove that there exists $\lambda$ such that $\psi_{\epsilon,\lambda}$ is a global solution, i.e. defined on $(0, \kappa^*]$,
and that $ \psi_{\epsilon,\lambda}(\kappa^*) = U(f(\kappa^*))/\rho$.
\begin{lemma} \label{lem:secHJ:philambda}
  We make  Assumptions \ref{ass:secHJ:1} and \ref{ass:secHJ:3} and suppose furthermore that
  $\delta=0$. 
  For every $(\epsilon,\lambda)\in\Omega^\uparrow$ with $0<\epsilon<\kappa^*$,
  there exists a unique maximal solution of the Cauchy problem \eqref{eq:chap:MFG_model:secHJ:19}-\eqref{eq:chap:MFG_model:secHJ:21}
  of the form $\Bigl( (0,\overline{k}(\epsilon,\lambda)), \psi_{\epsilon,\lambda} \Bigr)$  where $\epsilon< \overline{k}(\epsilon,\lambda)\le \kappa^*$. The function $\psi_{\epsilon,\lambda}$ is strictly concave and increasing in $(0,\overline{k}(\epsilon,\lambda))$.
\end{lemma}
\begin{proof}
  Existence and uniqueness of a maximal solution follow from the Cauchy-Lipschitz theorem.
  The strict monotonicity and concavity of  $\psi_{\epsilon,\lambda}$ are obtained as in Proposition \ref{prop:secHJ:vlambda}.
  Assume by contradiction that the interval in the definition of  the maximal solution is 
not of the form $(0,\overline{k}(\epsilon,\lambda))$. This implies that there exists $\underline{k}\in (0,\epsilon)$ such that either $\lim_{k\to \underline k} \psi_{\epsilon,\lambda}(k)=-\infty$ or $\psi_{\epsilon,\lambda}(\underline k)=U(f(\underline k))/\rho$. Let us rule out both situations:
  \begin{itemize}
  \item  If $\lim_{k\to \underline k} \psi_{\epsilon,\lambda}(k)=-\infty$, then
    $\lim_{k\to \underline k} \psi'_{\epsilon,\lambda}(k)=+\infty$. This implies that
  $\lim_{k\to \underline k} \psi_{\epsilon,\lambda}(k) =U(f(\underline k))/\rho$, and we have obtained the desired contradiction.
\item If  $\psi_{\epsilon,\lambda}(\underline k)=U(f(\underline k))/\rho$, then  proceeding as in the end of the proof of Proposition \ref{prop:secHJ:vlambda},  this implies that $\psi_{\epsilon,\lambda}( k)\le U(f( k))/\rho$ for all $k\in [\underline k,\epsilon]$, in contradiction with  $ \psi_{\epsilon,\lambda}(\epsilon)=\lambda>  U(f( \epsilon))/\rho$.
  \end{itemize}
  Therefore the maximal solution is defined in an interval of the form  $(0,\overline{k}(\epsilon,\lambda))$. 
\end{proof}

\begin{remark}\label{rem:4-2}
   Note that if $f(0)=0$, then $ \psi'_{\epsilon,\lambda}(k)$ blows up when $k\to 0^+$: indeed,  from \eqref{eq:chap:MFG_model:secHJ:8},
$    0 < H_q(k,\psi'_{\epsilon,\lambda}(k)) = f(k) - c^*\left( \psi'_{\epsilon,\lambda}(k)\right)$,    
  hence $ c^*\left( \psi'_{\epsilon,\lambda}(k)\right) <f(k)$. Therefore,
  $U'\left(c^*\left( \psi'_{\epsilon,\lambda}(k)\right) \right)> U'(f(k))$. Thus, from  Assumption \ref{ass:secHJ:1},
   $  \psi'_{\epsilon,\lambda}(k)=U'\left( c^*\left( \psi'_{\epsilon,\lambda}(k)\right)\right) > U'(f(k))$
tends to  $+\infty$ as $k\to 0$.
\end{remark}

\begin{lemma} \label{lem:secHJ:Lambda}
  Under the same assumptions as in Lemma~\ref{lem:secHJ:philambda}, for every $\epsilon \in (0, \kappa^*)$,  the set 
\begin{equation}\label{eq:chap:MFG_model:secHJ:22}
  \Lambda_\epsilon=\left\{\lambda>U(f( \epsilon))/\rho \quad \hbox{such that}\quad
    \overline{k}(\epsilon,\lambda)=  \kappa^*\right\}
\end{equation}
is not empty.
\end{lemma}
\begin{proof}
  Take $\lambda>U(f(\kappa^*))/\rho$. Assume by contradiction that $\overline{k}(\epsilon,\lambda)<\kappa^*$, where 
$\Bigl( (0,\overline{k}(\epsilon,\lambda)), \psi_{\epsilon,\lambda}   \Bigr)$ is the maximal solution 
of  the Cauchy problem \eqref{eq:chap:MFG_model:secHJ:19}-\eqref{eq:chap:MFG_model:secHJ:21}, (note that $\epsilon<\overline{k}(\epsilon,\lambda)$).
\\
Observe first that $\psi_{\epsilon,\lambda}$ cannot blow up as $k\to \overline{k}(\epsilon,\lambda)$. 
Indeed $v\mapsto \cF^\uparrow( k,\rho v)$ is Lipschitz continuous on $[ \max_{k\in [\epsilon, \kappa^*] } U(f(k))+1, +\infty) $ 
with a Lipschitz constant that does not depend on $k\in [\epsilon, \kappa^*]$. This property prevents $\psi_{\epsilon,\lambda}$  from blowing up in finite time.
\\
Therefore, the function
$\psi_{\epsilon,\lambda}$ can be extended to $\overline{k}(\epsilon,\lambda)$ by continuity, and
\begin{equation}
  \label{eq:chap:MFG_model:secHJ:maximal}
  \psi_{\epsilon,\lambda}( \overline{k}(\epsilon,\lambda)) = U(f(\overline{k}(\epsilon,\lambda)))/\rho,
\end{equation}
  otherwise it would not be the maximal solution.
 On the other hand,  we know that $f$ is increasing in $(0,\kappa^*]$, hence $U(f(\kappa^*))> U(f(k))$ for all $k<\kappa^*$.
In particular, $ U(f(\kappa^*))> U(f(\overline{k}(\epsilon,\lambda)))$.
 From the monotonicity of $\psi_{\epsilon,\lambda}$, we obtain that
 \begin{displaymath}
\psi_{\epsilon,\lambda} ( \overline{k}(\epsilon,\lambda)) \ge
 \psi_{\epsilon,\lambda} ( \epsilon)=\lambda > U(f(\kappa^*))/\rho >  U(f(\overline{k}(\epsilon,\lambda)))/\rho,   
 \end{displaymath}
which contradicts \eqref{eq:chap:MFG_model:secHJ:maximal}.
  \\
  We have proved that if $\lambda>U(f(\kappa^*))/\rho$, then the maximal solution is defined on $(0,\kappa^*]$. Therefore, $\Lambda_\epsilon$ is not empty.
\end{proof}
\begin{proposition}\label{prop:secHJ:psilambda}
  For all $\epsilon<\kappa^*$, there exists $\lambda$ such that $(\epsilon, \lambda) \in \Omega^\uparrow$ and 
   a global solution  $\psi_{\epsilon,\lambda}$  (i.e. defined on $(0,\kappa^*]$) of the Cauchy problem \eqref{eq:chap:MFG_model:secHJ:19}-\eqref{eq:chap:MFG_model:secHJ:21}  such that $\psi_{\epsilon,\lambda}(\kappa^*) = U(f(\kappa^*))/\rho$.
\end{proposition}
\begin{proof}
  Consider a decreasing sequence  $(\lambda_n)_{n\in\N}$ in $\Lambda_\epsilon$  (see (\ref{eq:chap:MFG_model:secHJ:22})) such that \\ $\lim_{n\to \infty} \lambda_n=  \underline \lambda_\epsilon=\inf_{\lambda \in  \Lambda_\epsilon} \lambda$.
  It is clear that $(\psi_{\epsilon, \lambda_n})_{n\in \N}$ is a decreasing sequence of functions defined on $(0,\kappa^*]$.
 Moreover, since $(k, \psi_{\epsilon, \lambda_n}(k))\in \Omega^\uparrow$ for $k\in (0, \kappa^*)$,  $\psi_{\epsilon, \lambda_n}$
 is bounded from below by the function $U\circ f /\rho$. Hence, there exists a function $\psi_\epsilon$ defined on $(0,\kappa^*]$ such that $\lim_{n\to +\infty} \psi_{\epsilon, \lambda_n}(k)=\psi_\epsilon(k)$ for all $k\in (0,\kappa^*]$.
\\
  Since $(\psi_{\epsilon, \lambda_n})_{n\in\N}$ is a sequence of concave functions locally uniformly bounded, 
 \cite[Theorem 3.3.3]{MR2041617} ensures that the convergence is uniform on every compact set, thus
 $\psi_\epsilon$ is continuous on $(0,\kappa^*]$.
Extending $\cF^\uparrow(\cdot,\cdot)$ by continuity on the set $\{ (k, U(f(k))/\rho): k\in (0,\kappa^*] \}$,
 we may pass to the limit in the integral form of the differential equation satisfied by $\psi_{\epsilon,\lambda_n}$ and get
  \begin{displaymath}
    \psi_{\epsilon}(k) = \underline{\lambda}_\epsilon +\int_\epsilon^k\cF^\uparrow(\kappa,\psi_{\epsilon}(\kappa))d\kappa.
  \end{displaymath}
 Hence $ \psi_{\epsilon}$ is a solution of \eqref{eq:chap:MFG_model:secHJ:19} on $(0,\kappa^*)$, which implies that $ \psi_{\epsilon}$ is $C^1$  and increasing in $(0,\kappa^*)$.
  \\
  We are left with proving that $\psi_{\epsilon}(\kappa^* )= U(f(\kappa^*))/\rho$. 
It is  already known that $\psi_{\epsilon}(\kappa^* )\ge U(f(\kappa^*))/\rho$. Assume by contradiction that  $\psi_{\epsilon}(\kappa^* )> U(f(\kappa^*))/\rho$. Then, set
  \begin{displaymath}
    b=\frac {\psi_{\epsilon}(\kappa^* )+  U(f(\kappa^*))/\rho} 2,
  \end{displaymath}
and  consider the Cauchy problem on $(0,\kappa^*]$:
\begin{eqnarray*}
  \xi'(k) &=& \cF^\uparrow (k,\xi(k)),
\\
(k,\xi(k))&\in& \Omega^\uparrow,\\
  \xi(\kappa^*) &=& b.
\end{eqnarray*}
It can be proved by contradiction (with the same kind of argument as in the end of the proof of Proposition \ref{prop:secHJ:vlambda})
that the maximal solution of this problem is in fact global, therefore defined on $(0,\kappa^*]$. Cauchy-Lipschitz theorem  implies that $\xi(k)< \psi_\epsilon(k)$ for all $k\in (0,\kappa^*]$. Therefore, $\xi(\epsilon)\in \Lambda_\epsilon$ and $\xi(\epsilon)< \psi_\epsilon(\epsilon)=\underline{\lambda}_\epsilon$, which contradicts the definition of $\underline{\lambda}_\epsilon$.\\
Therefore,  $\psi_{\epsilon}(\kappa^* )= U(f(\kappa^*))/\rho$. The same arguments as in the proof of Proposition \ref{prop:secHJ:vlambda} yield that
$\psi_{\epsilon}(k )>U(f(k))/\rho$ for all $k\in (0, \kappa^*)$. Hence $\psi_{\epsilon}=\psi_{\epsilon,\underline \lambda_\epsilon }$.
This achieves the proof.
\end{proof}
\begin{proposition} \label{prop:secHJ:ex_up}
 Under the same assumptions as in Lemma~\ref{lem:secHJ:philambda},
the Cauchy problem  (\ref{eq:chap:MFG_model:secHJ:ODEup1})-(\ref{eq:chap:MFG_model:secHJ:ODEup3}) has a unique solution 
 $u^\uparrow  \in C^1( (0,\kappa^*])\cap C^2(0,\kappa^*)$. Moreover $u^\uparrow$ is  strictly concave on $(0,\kappa^*)$.
\end{proposition}
\begin{proof}
  Existence is a consequence of Proposition \ref{prop:secHJ:psilambda}. Uniqueness is proved exactly with the same arguments as in the proof of Proposition \ref{prop:secHJ:ex_down}.
\end{proof}
\begin{remark}\label{rem:4-3}
  From Remark \ref{rem:4-2}, it is possible that $\lim_{k\to 0} \frac{du^\uparrow}{dk}(k)=+\infty$ and that $
\lim_{k\to 0} u^\uparrow(k)=-\infty$.
\end{remark}
\subsubsection{ End of the proof of Theorem~\ref{th:secHJ:main} in the particular case where
  $\delta=0$}
\label{sec:proof-theor-refth:s}
$\;$
\paragraph{Existence}
 With $u^\uparrow$ and $u^\downarrow$ as in Propositions \ref{prop:secHJ:ex_up} and \ref{prop:secHJ:ex_down},
define 
\begin{equation}  \label{eq:chap:MFG_model:secHJ:28}
u(k) = \begin{cases}
&u^\uparrow(k),\quad\quad  \hbox{if } k\in (0,\kappa^*],\\
&u^\downarrow(k), \quad\quad  \hbox{if } k\in[\kappa^*,+\infty).
\end{cases}  
\end{equation}
The properties of    $u^\uparrow$ and $u^\downarrow$ ensure that
$u$ is of class $C^1$, increasing and strictly concave in $(0,+\infty)$, and  $C^2$   in
$(0,\kappa^*)\cup (\kappa^*, +\infty)$. In particular,
$  u^\uparrow(\kappa^*)= u^\downarrow(\kappa^*)=\frac{1}{\rho} U(f(\kappa^*))$ and 
$ \frac {du^\uparrow}{dk} (\kappa^*)=\frac { du^\downarrow}{dk}(\kappa^*)=U'(f(\kappa^*))$. 
 Moreover,
\begin{displaymath}
  H_q\left(k,u'(k)\right) \begin{cases}
&>0,\quad\quad  \hbox{if } k\in (0,\kappa^*),\\
&<0, \quad\quad  \hbox{if } k\in (\kappa^*,+\infty),\\
&= 0 \quad\quad  \hbox{  if } k= \kappa^*.
\end{cases}  
\end{displaymath}
Hence, $u$ satisfies  (\ref{eq:chap:MFG_model:secGeneralities:HJ})-(\ref{eq:chap:MFG_model:secHJ:3}).

\bigskip

\paragraph{Uniqueness and characterization by (\ref{eq:chap:MFG_model:secGeneralities:valueFunction})}
Let us now prove that if \\
 $u\in C^1(0,+\infty)\cap C^2(  (0,\kappa^*) \cup (\kappa^* , +\infty))  $  satisfies  (\ref{eq:chap:MFG_model:secGeneralities:HJ})-(\ref{eq:chap:MFG_model:secHJ:3}), then it is the 
value function of problem (\ref{eq:chap:MFG_model:secGeneralities:2}). This will yield the uniqueness of a classical solution of  (\ref{eq:chap:MFG_model:secGeneralities:HJ})-(\ref{eq:chap:MFG_model:secHJ:3}) as well as the characterization of the value function of  (\ref{eq:chap:MFG_model:secGeneralities:2}).
\\
Let us set $\chi(\cdot)= c^* (u'(\cdot))  =f(\cdot) -H_q(\cdot, u'(\cdot))$.  
Assumptions  \ref{ass:secHJ:1}, \ref{ass:secHJ:3},  and
Lemma \ref{lem:secFP:1} below imply that $k\mapsto H_q(\cdot, u'(\cdot))$ is locally Lipschitz continuous on $(0,+\infty)$.
This property and (\ref{eq:chap:MFG_model:secHJ:2})-(\ref{eq:chap:MFG_model:secHJ:3}) imply that  for any $k_0\in (0,+\infty)$, 
there is a  unique solution $k$ of the Cauchy problem
  \begin{equation*}
    \begin{split}
    \frac{dk}{dt}(t) &=  f\left(k(t)\right)- \chi\left(k(t)\right) ,\quad t>0\\
    k(0)&=k_0,      
    \end{split}
  \end{equation*}
It is an admissible trajectory for problem (\ref{eq:chap:MFG_model:secGeneralities:2}). Therefore $u$ is not greater than
the value function of the optimal control problem  (\ref{eq:chap:MFG_model:secGeneralities:valueFunction}).
\\
On the other hand, consider     $c   \in   L^1_{{\rm loc}}( \R_+; \R_+)$, $\ell\in  L^1_{{\rm loc}}(\R_+; \R_+^d)$, $k\in W^{1,1}_{\rm{loc}} ( \R_+) $,
such that 
\begin{equation*}
  \begin{split}
    	\frac{dk}{dt}(t) = F(k(t),\ell(t))-w\cdot \ell(t) -\delta k(t)-c(t), \quad   \hbox{ for a.a. }t>0,
\\
k(0)= k_0,
\\
k(t)\geq 0, \quad   \hbox{ for a.a. }t>0.
  \end{split}
\end{equation*}
Observe that for almost every $t\geq 0$, 
\begin{equation*}
  \begin{split}
    	&\sup_{\bar{c}\geq 0,\bar{l}\geq 0}\left\{U(\bar{c})+u'(k(t))\left( F(k(t),\bar{l})-w\bar{l}-\delta k(t)- \bar{c}\right)\right\} 
  \\  
   \ge & U(c(t))+u'(k(t))\left( F(k(t),\ell(t))-w\ell(t)-\delta k(t)- c(t)\right) \\ =& U(c(t))+u'(k(t)) \frac{dk}{dt}(t).
  \end{split}
\end{equation*}
The left hand side coincides with $H\left(k(t), u'(k(t))\right)= \rho u (k (t))$. Hence,
$U(c(t) ) \le - u'(k(t)) 	\frac{dk}{dt}(t)  +   \rho u (k (t))$.
This implies that $\int_0^\infty U(c(t)) e^{-\rho t} dt \le u(k_0)$.  Hence, $u$ is
not smaller than the value function of problem   (\ref{eq:chap:MFG_model:secGeneralities:valueFunction}).
\\
We have proved that if  $u\in C^1(0,+\infty)$  satisfies  (\ref{eq:chap:MFG_model:secGeneralities:HJ})-(\ref{eq:chap:MFG_model:secHJ:3}), then it is the 
value function of problem (\ref{eq:chap:MFG_model:secGeneralities:2}).

\subsection{The case where $\delta>0$} 
\label{sec:proof-theor-refth:s-1}

$\;$

\begin{lemma}\label{lem:secHJ:ext-1}
  We make  Assumption \ref{ass:secHJ:3} and suppose furthermore that  $\delta>0$.
  Then there exits a unique $k_0\in(0,+\infty)$ such that 
  \begin{equation}
    \label{eq:6}
f(k_0) = 0.    
  \end{equation}
The function  $f$ takes  positive values on $(0,k_0)$ and negative values on $(k_0,+\infty)$. Moreover,
$f'(k_0)<0$ and $\kappa^*<k_0$, where $\kappa^*$ is the unique positive number such that $f'(\kappa^*)=\rho$, see (\ref{eq:1}).
\end{lemma}
\begin{proof}
  Since the proof is elementary, we skip it for brevity.
\end{proof}

\medskip

\begin{proof}[ Proof of Theorem \ref{th:secHJ:main} when  $\delta>0$]
Lemma \ref{lem:secHJ:ext-1} implies that  in the interval $(0,k_0)$
 which contains $\kappa^*$ and where $f$ is positive, it is possible to repeat the construction
 done in paragraphs \ref{sec:cauchy-probl-refeq:c} and \ref{sec:cauchy-probl-refeq:c2}. New arguments will be needed to construct the solution in $[k_0,+\infty)$.

\medskip

{\bf Step 1.} In $(0,k_0)$, it is possible to repeat the construction made  in paragraphs \ref{sec:cauchy-probl-refeq:c} and \ref{sec:cauchy-probl-refeq:c2}: 
there exists a unique classical solution  $u_1\in C^1(0, k_0)$ of the following problem:
  \begin{eqnarray}
    \label{eq:chap:MFG_model:secHJ:ext-2}
    -\rho u_1(k) + H\left(k,u_1'(k)\right)&=&0,  \quad \hbox{for } 0<k<k_0,\\
\label{eq:chap:MFG_model:secHJ:ext-3}    H_q\left(k,u_1'(k)\right)&>& 0,\quad    \hbox{for } 0<k< \kappa^*,\\
\label{eq:chap:MFG_model:secHJ:ext-4}    H_q\left(k,u_1'(k)\right)&<& 0,\quad   \hbox{for } \kappa^* < k <k_0.
  \end{eqnarray}
  The function  $u_1$ is  strictly concave and increasing in $(0,k_0)$.\\
  Since f is continuous and concave and $\lim_{k\to 0} f'(k)=+\infty$, $f'(k_0)<0$,  there exists  $\bar{k}\in(\kappa^* ,k_0)$ such that $    f(\bar{k}) = \max_{k\in[0,k_0]}f(k)$.  Since $u_1(\cdot)$ is increasing, $ \lim_{k\to k_0} u_1(k_0)\ge u_1(\bar k)$. On the other hand, $  \rho u_1(\bar k)>U (f(\bar k))$ (see paragraph \ref{sec:cauchy-probl-refeq:c}). Since $U$ is increasing, $U (f((\bar k))> \lim_{k\to k_0} U(f(k))= \lim_{c\to 0} U(c)$ (which may be $-\infty$). Therefore,
  \begin{equation*}
     \rho\lim_{k\to k_0} u_1(k_0)> \lim_{c\to 0} U(c).
  \end{equation*}
    With the same kind of arguments as in the proof of Proposition \ref{prop:secHJ:vlambda},
  it can also be  proved that $\rho u_1(k_0)< \lim_{c\to +\infty} U(c)$.
  This implies that $u_1(\cdot)$ can be extended by continuity to $(0,k_0]$ and that
  \begin{equation}
    \label{eq:7}
  	\lim_{c\to 0} U(c)< \rho u_1(k_0)<  \lim_{c\to \infty} U(c) .
  \end{equation}
The function $u'_1(\cdot)$ can then be extended by continuity to $k=k_0$ and \eqref{eq:chap:MFG_model:secHJ:ext-2} holds up to $k=k_0$.

\medskip

{\bf Step 2.} We are left with constructing the solution in $(k_0,+\infty)$.\\
Observe first that,
for any $k\ge k_0$, $q\mapsto H(k,q)$ is decreasing  from  \eqref{eq:chap:MFG_model:secHJ:8}, and that
\begin{enumerate}
\item $\lim_{q\to 0} H(k,q)=  \lim_{c\to +\infty} U(c)$ 
\item   Since $\lim_{q\to +\infty} c^* (q)= 0$ and $U(c)-cq +f(k)q \le U(c)$, we deduce that 
\begin{displaymath}
	\lim_{q\to +\infty} H(k,q) \le \lim_{c\to 0} U(c).
\end{displaymath}
\end{enumerate}
Hence, for any $k\ge k_0$, $q\mapsto H(k,q)$ maps $(0,+\infty)$
onto the interval \\ $ \left(\lim_{c\to 0} U(c), \lim_{c\to +\infty} U(c)\right)$
and has a right inverse $z\mapsto \cF(k,z)$: \\
for any $z\in  \left(\lim_{c\to 0} U(c), \lim_{c\to +\infty} U(c)\right)$, there is a unique $\cF(k,z)>0$ such that $H(k, \cF(k,z)) =z$.


\medskip

Let $\varepsilon>0$  be  small enough  so that $ \rho (u_1(k_0)-\varepsilon) >  \lim_{c\to 0} U(c) $, see (\ref{eq:7}). Set
\begin{equation}
\Omega=\left\{(k,v) \;: \;  k_0 \leq  k \text{ and } \rho (u_1(k_0)-\varepsilon) <\rho v<\lim_{c\to +\infty} U(c) \right\}.  
\end{equation}
Note that $(k_0,u_1(k_0))\in\Omega$.
It is possible to prove that $\cF(\cdot,\cdot)$ is of class $C^1$ on $\Omega$.
Furthermore, it can be seen that  $v\mapsto \cF(k,v)$ is Lipschitz continuous on $[u_1(k_0)-\varepsilon,\lim_{c\to \infty }U(c)/\rho]$
 with a Lipschitz  constant which does not depend on $k\in[k_0,+\infty)$.

Consider the Cauchy problem
\begin{eqnarray}\label{eq:chap:MFG_model:secHJ:ext-5}
  u_2'(k) &=& \cF (k,u_2(k)),\quad \quad\quad \hbox{for } k\ge k_0,\\
\label{eq:chap:MFG_model:secHJ:ext-6}
  (k,u_2(k))&\in& \Omega,\\
\label{eq:chap:MFG_model:secHJ:ext-7}
  u_2(k_0) &=& u_1(k_0).
\end{eqnarray}
From Cauchy-Lipchitz theorem, there is a unique maximal solution of \eqref{eq:chap:MFG_model:secHJ:ext-5}-\eqref{eq:chap:MFG_model:secHJ:ext-7}.
 The same arguments as in the proof of Proposition \ref{prop:secHJ:vlambda} yield that the solution is indeed global, 
i.e. defined on  $[k_0,+\infty)$, increasing and strictly concave.

\medskip

{\bf Step 3.}
  Set
\begin{equation*}
u(k) =	\begin{cases}
u_1(k),\quad \quad \text{if }k\in(0,k_0],\\
u_2(k),\quad  \quad \text{if }k\in [k_0,+\infty). 
\end{cases}  
\end{equation*}
From what precedes, $u\in C^1(0,+\infty)$, and   $\rho u(k) = H(k,u'(k))$ for any $k\in(0,+\infty)$.
  Note that $u$ is also $C^2$ in $(0,\kappa^*)\cup (\kappa^*,+\infty)$.
  Hence, $u$ is a classical solution of  (\ref{eq:chap:MFG_model:secGeneralities:HJ})-\eqref{eq:chap:MFG_model:secHJ:3}.
The remaining part of the proof (uniqueness and verification result) is exactly as in paragraph~\ref{sec:proof-theor-refth:s}.
\end{proof}

\section{The distribution of capital}\label{sec:secGeneralities:fokk-planck-equat}

We still assume that $w$,  the prices of the production factors, is a fixed vector in $(0,+\infty)^d$; we keep omitting $w$ everywhere.
The optimal investment policy of a firm with capital $k$ is $H_q(k, u'(k))$, where $u$ is the solution of  \eqref{eq:chap:MFG_model:secGeneralities:HJ}-\eqref{eq:chap:MFG_model:secHJ:3}. We are interested in finding a weak solution $m$  of the following problem:
\begin{eqnarray}
  \label{eq:13}
  \frac{d}{d k}\left(m H_q\left(\cdot, \frac{du}{d k}(\cdot) \right)\right) = \eta(\cdot, u (\cdot)) - \nu m(\cdot),
\\
  \label{eq:14}
     \nu \int_{\R_+}  m(k) dk= \int_{\R_+} \eta(k, u(k))  dk.
\end{eqnarray}
From \eqref{eq:chap:MFG_model:secHJ:2}-\eqref{eq:chap:MFG_model:secHJ:3}, we see that if
\eqref{eq:13} holds, then the optimal investment strategy has the effect of pushing $m$ toward $\kappa ^ *$.
It is therefore important to understand whether $m$ has a singularity at $k=\kappa ^ *$. For that,
the following lemma gives information on the behavior of $u$ near $\kappa ^ *$:
\begin{lemma}
  \label{lem:secFP:1}
   Under Assumptions \ref{ass:secHJ:1} and \ref{ass:secHJ:3}, there exist $\epsilon>0$ and $M>0$ such that
\begin{eqnarray}
    \label{eq:chap:MFG_model:FP1}
    0\leq &H_q(\kappa,u'(k)) \leq &M(\kappa^*-k), \quad \quad  \hbox{ if } \quad k\in [ \kappa^*-\epsilon, \kappa^*],\\
    \label{eq:chap:MFG_model:FP2}
     M(\kappa^*-k)\leq &H_q (k,u'(k)) \leq &0, \quad   \quad \quad \quad \quad \quad \hbox{ if } \quad   k \in [\kappa^*, \kappa^* +\epsilon].
  \end{eqnarray}
\end{lemma}

\begin{proof}
  We focus on the proof of  \eqref{eq:chap:MFG_model:FP1}, since the proof of  \eqref{eq:chap:MFG_model:FP2} is completely similar.
\\
Since $u\in C^1(0,+\infty)$, and $u$ is  $C^2$ in  $(0,\kappa ^*)\cup (\kappa^*,+\infty)$,
it is possible to take the derivative of \eqref{eq:chap:MFG_model:secGeneralities:HJ} at $\kappa\not= \kappa^ *$:
  \begin{equation}
    \label{eq:chap:MFG_model:FP3}
    \rho u'(\kappa) -  H_k(\kappa,u'(\kappa)) = H_q(\kappa,u'(\kappa))u''(\kappa).
  \end{equation}
  Let us set 
  \begin{equation}
     \label{eq:21}
     \chi (\kappa)= c^* (u'(\kappa)).
   \end{equation}
Note that $\chi(\kappa^*)=f(\kappa^*)$.
   The function $\chi$ is positive,  continuous and increasing  in $(0,+\infty)$, and $C^1$ on $(0,\kappa^*)\cup (\kappa^*,+\infty)$.   Recall that
\begin{displaymath}
  H_k(\kappa,u'(\kappa)) = f'(\kappa)u'(\kappa) , \quad        u'(\kappa)=U'(\chi(\kappa)), \quad   \hbox{ and } \quad  H_q(\kappa,u'(\kappa)) = f(\kappa) - \chi(\kappa).
\end{displaymath}	  
 Then \eqref{eq:chap:MFG_model:FP3} can be written as follows:
 \begin{equation}
    \label{eq:chap:MFG_model:FP4}
    U'(\chi(\kappa))(\rho - f'(\kappa)) = (f(\kappa) - \chi(\kappa))U''(\chi(\kappa)) \chi'(\kappa).
  \end{equation}

  \medskip
  
    The  inequality on the left hand side of \eqref{eq:chap:MFG_model:FP1} is already known since  $f(k) - \chi(k)>0$ for $k <\kappa^*$.
     We are left with proving the other inequality for $k$ sufficiently close to $\kappa^*$. \\
     We first claim that there exist
    $\epsilon>0$ and  $C_2>0$   such that for every $k \in [\kappa^* -\epsilon, \kappa^*]$,  
    \begin{equation}
    \label{eq:15}  \chi(\kappa^*)-\chi(k)=  f(\kappa^*)-\chi(k) \leq C_2(\kappa^*-k).
  \end{equation}
  \paragraph{Proof of  \eqref{eq:15}}
   For $0< \epsilon$ small enough, dividing  \eqref{eq:chap:MFG_model:FP4}   by $U''(\chi(\kappa))$ and integrating between $ k $ and $\kappa^*$ yields
    \begin{equation}
      \label{eq:chap:MFG_model:int_euler}
      \begin{split}
       & \int_k^{\kappa^*}\frac{U'(\chi(\kappa))}{U''(\chi(\kappa))}
        (\rho - f'(\kappa))d\kappa+
        \int_k^{\kappa^*}(f(\kappa^*)-f(\kappa))\chi'(\kappa)d\kappa \\
 =& \int_k^{\kappa^*} (\chi(\kappa^*)-\chi(\kappa))\chi'(\kappa)d\kappa= \frac{1}{2}(\chi(\kappa^*)-\chi(k))^2.
      \end{split}
\end{equation}
Let us deal with the first integral in the left hand side of   \eqref{eq:chap:MFG_model:int_euler}. 
Since $f\in W^{2,\infty}_{\rm loc}$, there exists $\epsilon_0>0$ and $C_0>0$ such that for all $k\in [\kappa^*-\epsilon_0, \kappa^*]$,
\begin{displaymath}
\rho - f'(\kappa)=  f'(\kappa^*)- f'(\kappa)=\int^{\kappa^*}_{\kappa}f''(z)dz \geq -C_0(\kappa^*-\kappa),
\end{displaymath}
thus
\begin{equation}
    \label{eq:chap:MFG_model:FP5}
   \int_k^{\kappa^*}\frac{U'(\chi(\kappa))}{U''(\chi(\kappa))}(\rho - f'(\kappa))d\kappa \leq -C_0\int_k^{\kappa^*}\frac{U'(\chi(\kappa))}{U''(\chi(\kappa))}(\kappa^*-\kappa)d\kappa
\end{equation}
Since $U'(\chi(\kappa))/ U''(\chi(\kappa))$ admits a negative limit as $\kappa\to \kappa^*$, there exists $C_1>0$ such that for all
$k\in [\kappa^*-\epsilon_0, \kappa^*]$,
\begin{equation}\label{eq:100000}
	\int_k^{\kappa^*}\frac{U'(\chi(\kappa))}{U''(\chi(\kappa))}(\rho - f'(\kappa))d\kappa\leq C_1(\kappa^*-k)^2.
\end{equation}
Next, integrating by part the second integral in  \eqref{eq:chap:MFG_model:int_euler} yields
 \begin{equation}
   \label{eq:chap:MFG_model:FP6}
  \begin{split}
  \int_k^{\kappa^*}(f(\kappa^*)-f(\kappa))\chi'(\kappa)d\kappa
  &=  \int_k^{\kappa^*}f'(\kappa) (\chi(\kappa)-\chi(k))  d\kappa\\
  &= (\chi(\kappa^*)-\chi(k)) \int_k^{\kappa^*}f'(\kappa) \frac {\chi(\kappa)-\chi(k)}
{ \chi(\kappa^*)-\chi(k)}
  d\kappa.
\end{split}
\end{equation}
Setting  $J(k) = \int_k^{\kappa^*}f'(\kappa)\frac{\chi(\kappa)-\chi(k)}{\chi(\kappa^*)-\chi(k)}d\kappa$, and using that 
both $f$ and $\chi$ are increasing, we obtain 
\begin{displaymath}
  0\leq J(k)\leq f(\kappa^*)-f(k). 
\end{displaymath}
Hence, there exists  $\epsilon_1>0$ and   $M_1>0$ and such that if 
\begin{equation}
  \label{eq:chap:MFG_model:FP8}
   0\leq J(k) \le M_1 ( \kappa^*-k),  \quad\quad \hbox{ for all } k\in [  \kappa^*-\epsilon_1, \kappa^*].
\end{equation}
From \eqref{eq:chap:MFG_model:int_euler}, 
\eqref{eq:100000}
and \eqref{eq:chap:MFG_model:FP6}, one deduces
that for $\epsilon\le \min(\epsilon_0,\epsilon_1) $,
\begin{equation}
  \label{eq:chap:MFG_model:proof_lem1_1}
  (\chi(\kappa^*)-\chi(k))^2 \leq
  2C_1(\kappa^*-k)^2+ 2(\chi(\kappa^*)-\chi(k))J(k) .
\end{equation}
Elementary algebra yields that for  all $k\in [\kappa^*-\epsilon, \kappa^*]$, 
\begin{equation}
  \label{eq:chap:MFG_model:FP9}
  \begin{split}
    0\le   \chi(\kappa^*)-\chi(k) &  \le  J(k)  + \Bigl(   J^2(k)+ 2C_1 (\kappa^*-k)^2   \Bigr)^\frac{1}{2} \\
    &  \le  \Bigl(M_1  + \left(   M_1^2 +  2C_1     \Bigr)^\frac{1}{2}\right) (\kappa^*-k),
  \end{split}
\end{equation}
where the last inequality is a consequence of  \eqref{eq:chap:MFG_model:FP8}.
The bound in  \eqref{eq:15}  is proved.

\bigskip

Finally, the definition of $\kappa^*$ in  \eqref{eq:1}
implies that  $f(k)-\chi(\kappa^*)=f(k)-f(\kappa^*) = -\rho(\kappa^*-k)+o(\kappa^*-k)$. 
Therefore, from  \eqref{eq:15}, there exists  $\epsilon>0$ and  $M>0$  such that for all  $k\in [\kappa^*-\epsilon, \kappa^*]$,   
\begin{displaymath}
0\le H_q(k,u'(k)) =	f(k) - \chi(k) \leq M(\kappa^*-k),
\end{displaymath}
which achieves the proof of  \eqref{eq:chap:MFG_model:FP1}.
\end{proof}

\begin{remark}
  Note that under the additional assumption that $f$ is locally uniformly concave,
 (i.e. for every compact set $K\subset (0,+\infty)$, there exists $\theta>0$ such that $
	f''(k)\leq -\theta$ for all $k\in K$),
it can be checked with a similar argument to the one in the  proof  of Lemma  \ref{lem:secFP:1} that there exists $\epsilon >0$ and $M_1>0$
such that for every $k\in[\kappa^*-\epsilon,\kappa^*+\epsilon]$, 
\begin{equation}
	\label{eq:chap:MFG_model:OS_estimate}
	|H_q(k,u'(k))|\geq M_1|\kappa^*-k|. 
\end{equation}
Consider $k\neq \kappa^*$ such that $|k-\kappa^*|\leq \epsilon$; by differentiating \eqref{eq:chap:MFG_model:secGeneralities:HJ} at  $k$, we obtain 
\begin{displaymath}
  u''(k) = \frac{u'(k)\left(\rho - f'(k)\right)}{H_q(k,u'(k))}.
\end{displaymath}
Using  estimate \eqref{eq:chap:MFG_model:OS_estimate} and the regularity of $f$, we deduce that there exists a constant $M_2>0$ independent of $k$
taken in $[\kappa^*-\epsilon,\kappa^*+\epsilon]$  such that
\begin{displaymath}
|u''(k)| \leq M_2 u'(k).
\end{displaymath}
This shows that $u''\in L^\infty(\kappa^*-\epsilon,\kappa^*+\epsilon)$. Finally, $u\in W^{2,\infty}_{\rm loc}(0,+\infty)$.
\end{remark}

\medskip

\begin{proof}[Proof of Proposition \ref{prop:FP_1}]
For brevity, we use the notation $b(k)= H_q(k,u'(k))$. If $m$ satisfies (\ref{eq:chap:MFG_model:FP}) in the sense of distributions
 and  (\ref{eq:12}), then the  weak derivative of $bm$  is $\eta(\cdot, u(\cdot))-\nu m$,
 a bounded measure from (\ref{eq:12}) and Assumption \ref{ass:secFP:1}. Hence $bm \in {\rm BV}_{\rm loc}(0,+\infty)$. On the other hand,
 $1/b\in C^1((0,\kappa^*)\cup(\kappa^*,+\infty))$. Therefore,  the restriction of $m$ to $(0,\kappa^*)\cup(\kappa^*,+\infty)$ 
can be written $(bm)/b$ and  identified with a function in ${\rm BV}_{\rm loc}(
(0,\kappa^*)\cup(\kappa^*,+\infty))$. The Lebesgue decomposition of $m$ is  $m=m_{ac}+m_s$;
the singular part $m_s$ is supported in $\{\kappa^*\}$, hence $m_s=\lambda \delta_{\kappa*}$ with $\lambda\ge 0$;
 the regular part $m_{ac}$ can be identified with a nonnegative function in $L^1(0,+\infty)$. \\
We claim that $\lambda=0$. To prove this fact, consider a family $(\varphi_\varepsilon)_{\varepsilon>0}$
 such that
\begin{itemize}
	\item $\varphi_\varepsilon \in C^\infty_c(0,+\infty)$
	\item ${\rm supp} (\varphi_\varepsilon) \subset [\kappa^*- \varepsilon,\kappa^*+\varepsilon]$
	\item $\varphi_\varepsilon(\kappa^*)=1$
	\item $\varphi_\varepsilon$ is non decreasing on $[0,\kappa^*]$, and non increasing in $[\kappa^*,+\infty)$
	\item $\|{\varphi_\varepsilon'}\|_\infty\leq 2/\varepsilon$
\end{itemize}
We deduce from  (\ref{eq:chap:MFG_model:FP})-(\ref{eq:12}) that for $\varepsilon$ small enough,
\begin{displaymath}
  -\int_{\R_+} \varphi_\varepsilon'(k)b(k)dm(k) = 
-\nu\int_{\R_+}  \varphi_\varepsilon(k) dm(k) + \int_{\R_+}  \varphi_\varepsilon(k) \eta(k, u(k)) dk.
\end{displaymath}
For  $\varepsilon\in(0, \kappa^*/2)$, this leads to
\begin{equation*}
  \begin{split}
  &-\int_{\kappa^* - \varepsilon} ^ {\kappa^* + \varepsilon}  \varphi_\varepsilon'(k)b(k)m_{ac}(k) dk \\=&
-\nu \int_{\kappa^* - \varepsilon} ^ {\kappa^* + \varepsilon} \varphi_\varepsilon(k) m_{ac}(k) dk 
+ \int_{\kappa^* - \varepsilon} ^ {\kappa^* + \varepsilon} \varphi_\varepsilon(k) \eta(k, u(k)) dk -\nu \lambda ,    
  \end{split}
\end{equation*}
 because $b(\kappa^*)=0$. 
The construction of $ \varphi_\varepsilon$ and (\ref{eq:chap:MFG_model:FP1})-(\ref{eq:chap:MFG_model:FP2}) ensure that
\begin{displaymath}
  \sup_{k\in [\kappa^*-\varepsilon,\kappa^*+\varepsilon]}  |\varphi'_\varepsilon(k) b(k)|\leq 2M.
\end{displaymath}
This yields
\begin{displaymath}
0\le   \nu \lambda\leq 2M\int_{k^*-\varepsilon}^{k^*+\varepsilon} m_{ac}(k)dk+ 
\int_{\kappa^* - \varepsilon} ^ {\kappa^* + \varepsilon} \varphi_\varepsilon(k) \eta(k, u(k)) dk.
\end{displaymath}
Letting $\varepsilon\to 0$, we obtain that $\lambda=0$ by applying Lebesgue dominated convergence theorem. The claim is proved.
\\
Therefore, $m\in L^1(0,+\infty)$, and   (\ref{eq:chap:MFG_model:FP1}) implies that $bm\in W^{1,1}_{\rm loc}(0,+\infty)$,
 and that
  $0\le m\in L^1(0,+\infty)\cap C^1((0,\kappa^*)\cup (\kappa^*,+\infty))$.
\\
Integrating  equation \eqref{eq:chap:MFG_model:FP}  over the intervals  $(0,\kappa^*)$ and 
 $(\kappa^*,+\infty)$, we see that 
\begin{equation*}
  \begin{split}
      & b(k)m(k)=\\
    &\left\{ \begin{array}[c]{rl}
              \ds    \int_0^{k}\eta(\kappa, u(\kappa))
              \exp\left(-\ds \int_\kappa^k\frac{\nu}{    	b(z)}dz\right)d\kappa    + A \exp\left(-\ds \int_{\frac {\kappa^*} 2}^k\frac{\nu}{  b(z)}dz\right)      ,&\text{ if }  0<k <\kappa^*,\\
              \ds  -\int_k^{\infty}\eta(\kappa, u(\kappa))\exp
               \left(\ds \int_k^\kappa\frac{\nu}{ b(z)}dz\right)d\kappa +
               B
 \exp\left(\ds \int_k^{\frac {3\kappa^*} 2} \frac{\nu}{ b(z)}dz \right) , &\text{ if }    k>\kappa ^*,
            \end{array}\right.
  \end{split}
\end{equation*}
for two real numbers $A$ and $B$. But, from Lemma \ref{lem:secFP:1}, we see that a necessary condition for the integrability of $m$ is  that $A=B=0$.
Imposing $A=B=0$, we see that $m$ is a nonnegative function.
It remains to check \eqref{eq:12}.
Set $I_1=\int_0^{\kappa^*}m(k)dk$ and      $I_2=\int_{\kappa^*}^{+\infty}m(k)dk$.\\
      Focusing on $I_1$,
      \begin{equation}
          \label{eq:chap:MFG_model:FP11}
          \begin{array}[c]{rcl}        I_1 &=& \ds \int_0^{\kappa^*}\frac{1}{b(k)}
              \int_0^k\eta(\kappa, u(\kappa))\exp\left(-\int^k_\kappa\frac{\nu}{b(z)}dz\right) d\kappa dk,\\
            &=& \ds \int_0^{\kappa^*}\eta(\kappa, , u(\kappa))\int_\kappa^{\kappa^*}\frac{1}{b(k)}\exp\left(-\int^k_\kappa\frac{\nu}{b(z)}dz\right)dkd\kappa, \\
            &=&\ds \frac{1}{\nu}\int_0^{\kappa^*}\eta(\kappa,  u(\kappa)) d\kappa.            
          \end{array}
        \end{equation}
The second line in \eqref{eq:chap:MFG_model:FP11} is obtained using the non negativity of the integrand and Tonelli's theorem. 
The third line in  \eqref{eq:chap:MFG_model:FP11} comes from the fact that
$	\int_\kappa^{\kappa^*}\frac{\nu}{b(z)}dz = +\infty$,
which is a consequence of Lemma \ref{lem:secFP:1}.
\\
It can be proved in the same way that
\begin{displaymath}
  I_2 = \frac{1}{\nu}\int_{\kappa^*}^{+\infty}\eta(\kappa, u(\kappa)) d\kappa.
\end{displaymath}
Hence $\nu (I_1+I_2)= \int_{\R_+}\eta(\kappa, u(\kappa)) d\kappa  $,
and $m$ given by (\ref{eq:chap:MFG_model:FP10}) is the unique  solution of 
\eqref{eq:chap:MFG_model:FP}-\eqref{eq:12}. 
\end{proof}


 \section{Equilibrium}
 \label{sec:equilibrium}

 This paragraph is devoted to existence of equilibria.
  

 \subsection{Stability results for   \eqref{eq:chap:MFG_model:secGeneralities:HJ}-\eqref{eq:chap:MFG_model:secHJ:3}}
\label{sec:equil1}
 
\begin{lemma}
  \label{lem2}
  Under Assumptions  \ref{ass:secHJ:1} and \ref{ass:secHJ:3}, the value function $-u$ is monotone with respect to $w$, i.e. for every $w,\tilde w \in(0,+\infty)^d$,
  \begin{displaymath}
    w \le \tilde w \quad \quad \Rightarrow \quad \quad  u(\cdot,w)\geq u(\cdot,\tilde w).
  \end{displaymath}
\end{lemma}
\begin{proof}
  Assume $w \le \tilde w$  and consider an admissible trajectory associated with the vector of prices $\tilde w$: it satisfies
  $\frac{dk}{dt}(t)=f(k(t),\tilde w) -c(t)$ with $k(0)=k_0$. The differential equation also reads:
$\frac{dk}{dt}(t)=f(k(t), w) -(  c(t) +f(k(t), w)- f(k(t),\tilde w))  $, and $ c(t) +f(k(t), w)- f(k(t),\tilde w)\ge c(t)\ge 0$, which can be used as a control. This yields that
\begin{displaymath}
  u(k_0,w)\ge \int_0^\infty U\left( c(t) +f(k(t), w)- f(k(t),\tilde w)\right) e^{-\rho t} dt \ge \int_0^\infty U( c(t) ) e^{-\rho t} dt .
\end{displaymath}
Taking the supremum on all admissible trajectories associated with $\tilde w$, we deduce that for all $k_0>0$, $u(k_0,w)\ge u(k_0,\tilde w)$.
  \end{proof}

\begin{lemma}
\label{lem:equilibrium:continuity_k*}
Under Assumption  \ref{ass:secHJ:3}, the map $(0,+\infty)^d   \ni w\mapsto \kappa^*(w)\in (0,+\infty)$ defined in \eqref{eq:1}
is continuous.
\end{lemma}

\begin{proof}
  Consider a sequence $(w_n)_{n\in \N}$, $w_n\in (0,+\infty)^d$, such that $w_n$ tends to $w\in (0,+\infty)^d$ as  $n\to +\infty$.\\
  We first claim that $\kappa^*(w_n)$ remains in a compact subset of $(0,+\infty)$. We proceed by contradiction:
  \begin{itemize}
  \item   Assume first that  up to the extraction of a subsequence,
  $\kappa^*(w_n)$ tends to $+\infty$ as $n\to +\infty$. Hence, for any $k>0$, there exists $N>0$ such that if $n\ge N$, then
  $\frac {\partial f}{\partial k} (k, w_n)> \rho$. Passing to the limit using the $C^1$ regularity of $f$ (see Assumption  \ref{ass:secHJ:3}),
  we get that  $\frac {\partial f}{\partial k} (k, w)\ge  \rho$ for all $k>0$. But $k\mapsto f(k,w)$ is strictly concave: Hence, $\frac {\partial f}{\partial k} (k, w)> \rho$ for all $k>0$.
  This contradicts  point  2.iii in  Assumption  \ref{ass:secHJ:3} (see also Remark \ref{rem:1}).
    \item  Assume that up to the extraction of a subsequence, $\kappa^*(w_n)$ tends to $0$ as $n\to +\infty$: arguing as above, this implies that
      $\frac {\partial f}{\partial k} (k, w)< \rho$ for all $k>0$.  This contradicts  point  2.ii in  Assumption  \ref{ass:secHJ:3}.
  \end{itemize}
  The claim is proved.\\
  Possibly after the extraction of a subsequence,  $\kappa^*(w_n)$ tends to a positive limit $\tilde \kappa$. It is easy to deduce from Assumption  \ref{ass:secHJ:3}
  that $\frac {\partial f}{\partial k} (\tilde \kappa, w)=\rho$. Therefore $\tilde \kappa=\kappa^* (w)$, and the uniqueness of the cluster point implies that the whole sequence
  $\kappa^*(w_n)$ tends to $\kappa^*(w)$. This achieves the proof.
\end{proof}

\begin{lemma}[Continuity of $w\mapsto u(\cdot,w)$]
  \label{lem:equilibrium:stability}
  Let  $(w_n)_{n\in\N}$,  $w_n\in (0,+\infty)^d$,  be a sequence converging to $w\in(0,+\infty)^d$ as $n\to \infty$.
  Then, under Assumptions  \ref{ass:secHJ:1} and \ref{ass:secHJ:3}, 
\begin{displaymath}
  u(\cdot,w_n)\rightarrow u(\cdot, w)
\end{displaymath}
in $C^1(K)$ for every compact subset $K$ of  $(0,+\infty)$.
\end{lemma}
\begin{proof}
  We may assume without loss of generality that  there exist  two vectors  $\underline{w},\overline{w}\in (0,+\infty)^d$ such that, for all $n\ge 0$,
  \begin{displaymath}
    \underline{w} \le w_n \le \overline{w}.
  \end{displaymath}
 From Lemma \ref{lem2}, the following inequalities hold  for all $n\ge 0$:
 \begin{displaymath}
   u(\cdot,\overline{w})\leq u(\cdot,w_n)\leq u(\cdot,\underline{w}).
 \end{displaymath}
 Using 	\eqref{eq:chap:MFG_model:secGeneralities:HJ} and the  coercivity of $q\mapsto H(k,q,w_n)$  uniform w.r.t. $n$ and 
 $k\in K$, where $K$ is a compact subset of   $ (0,+\infty)$,
 we see that $\frac{\partial u}{\partial k}(k,w_n)$ is bounded uniformly in $n$ and $k\in K$.  Moreover, if
 $\lim_{c\to +\infty} U(c)=+\infty$, then  $\frac{\partial u}{\partial k}(\cdot,w_n)$ is bounded uniformly away from $0$  w.r.t. $n$ and $k\in K$.
\\
 Since $(u(\cdot,w_n))_{n\in\N}$ is a sequence of concave functions uniformly bounded on every compact subset of $(0,+\infty)$,
 there exists  a continuous and concave function  $v:(0,+\infty)\rightarrow \R$ such that, after the extraction of a subsequence,
 \begin{itemize}
 \item $u(\cdot,w_n)\rightarrow v$  locally uniformly in $(0,+\infty)$
 \item $\frac{\partial u}{\partial k}(\cdot,w_n)\rightarrow v'$ almost everywhere  in $(0,+\infty)$.
 \end{itemize}
On the other hand,  from Lemma \ref{lem:equilibrium:continuity_k*}, there exist $\underline \kappa>0$ and $\overline \kappa> \underline \kappa$ such that
 \begin{displaymath}
   \underline \kappa <\min_{\underline{w}\leq w \leq \overline{w}}\kappa ^*(w)\leq \max_{\underline{w}\leq w \leq
     \overline{w}}\kappa^*(w)< \overline \kappa.
 \end{displaymath}
 Take any compact interval $[a,b]$ such that $0<a< \underline \kappa $ and $ \overline \kappa<b$.

 The functions $u(\cdot,w_n)$ are uniformly Lipschitz viscosity solutions (with $\frac{\partial u}{\partial k}(\cdot,w_n)$ bounded away from $0$ if  $\lim_{c\to +\infty} U(c)=+\infty$) of 
\eqref{eq:chap:MFG_model:secGeneralities:HJ} (with $w=w_n$) on $( a , b)$ with state constrained boundary conditions at $ a $ and $b$.
From the continuity of $H$ on $[ a ,b]\times(0,+\infty)^d\times (0,+\infty)^d$, the uniform bounds  on  $\frac{\partial u}{\partial k}(\cdot,w_n)$  stated above and
the uniform convergence of $(u(\cdot,w_n))_{n\in\N}$ towards $v$ on $[ a ,b ]$,
stability results on viscosity solutions, see e.g. \cite{MR1484411} can be used and yield that 
$v$ is a viscosity solution of
\begin{displaymath}
	\rho v(k) = H(k,v'(k), w),
\end{displaymath}
on $( a,b)$, with state constrained boundary conditions at  $k= a$ and $k=b$.
Note that the eventuality that $H(k,q,w)\to+\infty$ as $q\to 0$ does not imply any difficulty, because in this case,  $\frac{\partial u}{\partial k}(\cdot,w_n)$ is  uniformly bounded away from $0$. From this observation, we can also use well-known results on the uniqueness of state constrained solutions of the Hamilton-Jacobi equation, see e.g. \cite{MR1484411}, and find that $v=u(\cdot,w)$. 

In fact, the convergence holds locally in $C^1$. We know that 
\begin{itemize}
\item $u(\cdot,w_n)$ tends to $ u(\cdot,w)$ uniformly in $[a,b]$
\item there exists a measurable subset $E$ of $[a,b]$, such that the Lebesgue measure of $[a,b] \setminus E$ is zero and that
  $\frac{\partial u}{\partial k}	(\cdot,w_n)$ tends to $ \frac{\partial u}{\partial k} (\cdot,w)$ pointwise in $E$.
\end{itemize}
Note that after slightly modifying  $a$ or $b$ if necessary, we can always assume that   $a\in E$ and $b\in E$. A variant of Dini's first theorem  yields that the  convergence
of $ \frac{\partial u}{\partial k}	(\cdot,w_n)$ is in fact uniform in $[a,b]$: for completeness, the proof is given in what follows.\\
The function $  \frac{\partial u}{\partial k}(\cdot,w)$ is continuous, thus uniformly continuous  on $[a,b]$; hence, given $\epsilon>0$,
it is possible to choose $\delta>0$ small enough such that 
\begin{displaymath}
  |k-k'|\le \delta\quad  \Rightarrow \quad \left| \frac{\partial u}{\partial k}(k,w)- \frac{\partial u}{\partial k}(k',w)\right|<\frac{\varepsilon}{2},\quad\quad \forall k,k'\in[a,b].
\end{displaymath}
For such a choice of $\delta>0$, it is possible to define a finite
subdivision $(\sigma_i)_{i\in\{0,\dots,I\}}$ of $[a,b]$ such that
\begin{itemize}
	\item for every $i\in \{ 0,\dots,I \}$, $\sigma_i\in E$.
	\item           for any $i\in\{0,\dots,I-1\}$, $0<\sigma_{i+1}-\sigma_i<\delta$.
\end{itemize}
On the other hand, for any $k\in [a,b]$, there exists $i_0\in\{0,\dots I-1\}$ such that $\sigma_{i_0}\leq k\leq \sigma_{i_0+1}$. Then the concavity of $u$ with respect to $k$ yields
\begin{align*}
  \frac{\partial u}{\partial k}(k,w_n) - \frac{\partial u}{\partial k}(k,w)&\leq \frac{\partial u}{\partial k}(\sigma_{i_0},w_n) - \frac{\partial u}{\partial k}(\sigma_{i_0+1},w)\\
                         &=\frac{\partial u}{\partial k}(\sigma_{i_0},w_n) - \frac{\partial u}{\partial k}(\sigma_{i_0},w) + \frac{\partial u}{\partial k}(\sigma_{i_0},w) - \frac{\partial u}{\partial k}(\sigma_{i_0+1},w).
\end{align*}
Taking $N\in \N$ large enough such that for every $n\geq N$, 
\begin{displaymath}
  \max_{0\le i\le I }\left| \frac{\partial u}{\partial k}(\sigma_{i},w_n) - \frac{\partial u}{\partial k}u(\sigma_{i},w)\right| <\frac{\varepsilon}{2}
\end{displaymath}
yields that
\begin{displaymath}
  \frac{\partial u}{\partial k}(k,w_n) - \frac{\partial u}{\partial k}u(k,w) < \varepsilon,\quad \forall n\geq N.
\end{displaymath}
A similar argument can be used to bound $  \frac{\partial u}{\partial k}(k,w_n) -  \frac{\partial u}{\partial k}(k,w)$ from below. Finally, for any $ \varepsilon>0$ there exists
$N>0$ such that 
\begin{displaymath}
   \sup_{k\in[a,b]}\left| \frac{\partial u}{\partial k}(k,w_n) -  \frac{\partial u}{\partial k}(k,w)\right|< \varepsilon,\quad\quad  \forall n\geq N.
 \end{displaymath}
 This achieves the proof.

\end{proof}

\subsection{Existence of equilibria}\label{sec:equil2}

\begin{proof} [Proof of Theorem \ref{th_ex_equil}]
Recall that $\Phi$ and   $g$ are  respectively defined in Assumption  \ref{ass:chap:MFG_model:S} and  formula \eqref{eq:supply0}. Let $\epsilon$ be the constant appearing in Assumption
\ref{ass:chap:MFG_model:technical_assump}. There exist two constants $0<\underline{\kappa}\le \overline \kappa<+\infty$ such that for all $w \in [\epsilon, 1/\epsilon]^d$, $\underline \kappa < \kappa^* (w)< \overline \kappa$. Hence,
$m(\cdot, w)$ is supported in the compact inteval $ J= \hbox{conv} \left([
  \underline \kappa, \overline \kappa]\cap \hbox{support}(\hat \eta) \right)$.

We claim that the map $w\mapsto m(\cdot, w)$ is continuous from  $[\epsilon, 1/\epsilon]^d$ to the set of probability measures supported in $J$.
Indeed, let $(w_n)_{n\in\N}$,  $w_n\in   [\epsilon, 1/\epsilon]^d   $, be a sequence  converging to $w$ as $n\to +\infty$.  From Lemma  \ref{lem:equilibrium:stability},  $	u(\cdot,w_n)\to u(\cdot, w)$ in $C^1(K)$  for any compact subset $K$ of $(0,+\infty)$. The probability measures $m(\cdot, w_n) $ are all supported in $J$.
 Hence, the sequence $m(\cdot, w_n) $  has a cluster point $\mu$ in the weak $*$ topology. Let us prove that $\mu=m(\cdot, w)$: for any test function $\phi(\cdot)\in C^\infty_c(0,+\infty)$,
\begin{displaymath}
	-\int_0^{+\infty}\phi'(k)b(k,w_n)m(k,w_n)dk = \int_0^{+\infty}\phi(k)\eta(k)dk - \nu \int_0^{+\infty}\varphi(k)m(k,w_n)dk.
      \end{displaymath}
      where $b$ is given by \eqref{eq:b}.\\
  The right-hand side converges to $\ds
	\int_0^{+\infty}\phi(k)\eta(k)dk -\nu \int_0^{+\infty}\phi(k)\mu(k)dk$.
        On the other hand, the $C^1$ convergence of $u(\cdot,w_n)$ to $u(\cdot,w)$ on every compact subset of $(0,+\infty)$ implies the uniform convergence of $	H_q(\cdot,\frac{\partial u}{\partial k}\left(\cdot,w_n),w_n\right)$ to $H_q\left(\cdot,\frac{\partial u}{\partial k}(\cdot,w),w\right)$ in $J$.
        We deduce that
\begin{displaymath}
	\int_0^{\infty}\phi'(k)H_q\left(k,\frac{\partial u}{\partial k}(k,w_n),w_n\right)m(k,w_n)dk \to \int_0^{\infty}\phi'(k)H_q\left(k,\frac{\partial u}{\partial k}(k,w),w\right)\mu(k)dk.
      \end{displaymath}
      Therefore $\mu = m(\cdot,w)$ and the whole sequence $ m(\cdot,w_n)$ weakly $*$ converges to $m(\cdot,w)$ as $n\to \infty$. The map $w\mapsto m(\cdot,w)$ is continuous on  $[\epsilon, 1/\epsilon]^d$.

\medskip

For $\lambda\in [0,1]$, we then consider the map $T_\lambda$ defined on $[\epsilon, 1/\epsilon]^d$ by
\begin{equation}
  \label{eq:10001}
T_\lambda(w)=\argmin\left\{ \Phi(\cdot)  + \int_0^\infty g(k, \cdot)   \Bigl( (1-\lambda) d\hat \eta (k) + \lambda dm(k,w) \Bigr) \right\},
\end{equation}
where the  function $g$ has been defined in \eqref{eq:supply0}, (recall that $k\mapsto g(k)$ is convex).
From the observation made above on $m(\cdot, w)$ and from Assumption  \ref{ass:chap:MFG_model:S},
the function to be minimized is continuous, strictly convex  and coercive on $[0,+\infty)^d$; hence $T_\lambda(w)$ is well defined. Moreover, $\|T_\lambda(w)\|_\infty$ is bounded uniformly in $w\in  [\epsilon, 1/\epsilon]^d$.
\\
Let $w_n$ and $\lambda_n$ be two sequences taking their values respectively in $ [\epsilon, 1/\epsilon]^d$ and
in $[0,1]$; assume that $w_n$ tends to $w$ and that $\lambda_n$ tends to $\lambda$.  The sequence $T_{\lambda_n } (w_n)$ takes its values in a compact; hence, up to the extraction of a subsequence, we may assume that $T_{\lambda_n } (w_n)$ converges to some $\tilde w$. Since
$m(\cdot, w_n)$ weakly $*$ converges to $m(\cdot, w)$, it is easy to check that $\tilde w=T_\lambda(w)$ and that the whole sequence $ T_{\lambda_n } (w_n)$ converges. Hence, the map $(\lambda, w)\mapsto T_\lambda(w)$ is continuous.

\medskip

For $\lambda\in [0,1]$, we consider the equation: find $w\in [\epsilon, 1/\epsilon]^d$ such that   $w-T_\lambda(w)=0$, which we write $\chi(w,\lambda)=0$.  We now aim at applying Brouwer degree theory to $\chi$.  

First, setting  $t_0= \argmin\left\{ \Phi(\cdot)  + \int_0^\infty g(k, \cdot)d\hat \eta (k)\right\}$ which  does not depend on $w$,  the equation $\chi(w,0)=0$  writes $w=t_0\in (\epsilon, 1/\epsilon)^d$. Therefore,
  \begin{equation}
    \label{eq:10002}
  \hbox{deg}\left(\chi(\cdot, 0), (\epsilon, 1/\epsilon)^d, 0_{\R^d}\right)=1.     
  \end{equation}

  Second,  for all $\lambda \in [0,1]$, we know from Assumption  \ref{ass:chap:MFG_model:technical_assump} that  the equation  $w-T_\lambda(w)=0$ has no solution on the boundary of  $[\epsilon, 1/\epsilon]^d$.

  From the two observations above, we see that for all $\lambda\in [0,1]$,
   \begin{equation}
    \label{eq:10003}
  \hbox{deg}\left(\chi(\cdot, \lambda), (\epsilon, 1/\epsilon)^d, 0_{\R^d}\right)=1.     
  \end{equation}
  
  We deduce that there exists $w^*\in  (\epsilon, 1/\epsilon)^d$ such that
  \begin{displaymath}
   w^*= \argmin\left\{ \Phi(\cdot)  + \int_0^\infty g(k, \cdot)   dm(k,w^*) \right\}.
 \end{displaymath}
 Writing the  first order necessary optimality conditions associated with this minimization problem, we see that $w^*$ satisfies \eqref{eq:clearing_condition}.
\end{proof}

\begin{remark}
  We have actually proved more than the existence of an equilibrium, namely that  $\hbox{deg}(\chi, (\epsilon, 1/\epsilon)^d, 0)=1$.
\end{remark}
\subsection{Assumption   \ref{ass:chap:MFG_model:technical_assump} holds in the examples of Subsection \ref{sec:import-exampl-util}}
\label{sec:assumption_holds}

\subsubsection{The  Cobb-Douglas  production function} \label{sec:assump_cobb_douglas}
\begin{proposition}
  Assumption   \ref{ass:chap:MFG_model:technical_assump} holds with
  the Cobb-Douglas production function described in Subsection \ref{sec:import-exampl-util} .
\end{proposition}

\begin{proof}
From \eqref{eq:chap:examples:net_output_CobbDouglas}, we deduce that  for two positive  constants $c_1$ and $c_2$
  \begin{equation}
    \label{eq:10006}
    g(k,w)     = c_1 k^{\frac \alpha {1-|\beta|}}  G_\beta (w)\quad \hbox{and}\quad   \kappa^*(w)=c_2  \left(G_\beta (w)\right) ^{\frac {1-|\beta|}  {1-\alpha-|\beta|}},
\end{equation}
where
\begin{displaymath}
  G_\beta (w)= \prod_{i=1}^d w_i^{-\frac {\beta_i} {1-|\beta|}}.
\end{displaymath}
Setting
\begin{displaymath}
M_\lambda(w)=   \left( \lambda 
   c_1 \int_0^\infty  k^{\frac \alpha {1-|\beta|}} dm(k,w) + (1-\lambda) M_0\right)  \quad \quad \hbox{with} \quad 
 M_0 =  c_1\int_0^\infty  k^{\frac \alpha {1-|\beta|}} d\hat \eta(k),
\end{displaymath}
 \eqref{eq:10004} becomes
\begin{equation}
    \label{eq:10008}
  \Phi(w)+ M_\lambda (w) G_\beta(w) \le \Phi(\one)+M_\lambda(w) .
\end{equation}
Since $\Phi(w)\ge 0$,   \eqref{eq:10008} implies that $G_\beta(w) \le 1+   \Phi(\one)/ M_\lambda(w)$.
On the other hand, \eqref{eq:10006} yields
\begin{equation}
  \label{eq:10005}
M_\lambda(w)\ge c_1 \lambda \min \left(   \underline a, c_2  \left(G_\beta (w)\right) ^{\frac {1-|\beta|}  {1-\alpha-|\beta|}}\right) ^\frac \alpha {1-|\beta|}   +(1-\lambda) M_0 ,
\end{equation}
where $\underline a$ is the minimal value in the support of $\hat \eta$. Combining the latter two estimates yields
\begin{equation}
  \label{eq:10007}
  G_\beta(w) \le 1+  \frac { \Phi(\one) } {
 c_1  \lambda \min \left(   \underline a, c_2  \left(G_\beta (w)\right) ^{\frac {1-|\beta|}  {1-\alpha-|\beta|}}\right) ^\frac \alpha {1-|\beta|}  +(1-\lambda)M_0    }.
\end{equation}
It is easy to deduce from  \eqref{eq:10007} that $ G_\beta(w) < c_3 $, for a positive constant $c_3$ independent of $w$.
\\
If  $\overline a$ is the maximal value in the support of $\hat \eta$, this implies that
\begin{equation}
  \label{eq:10010}
  \begin{split}
    M_\lambda(w)& \le c_1 \lambda \max \left(   \overline a, c_2  \left(G_\beta (w)\right) ^{\frac {1-|\beta|}  {1-\alpha-|\beta|}}\right) ^\frac \alpha {1-|\beta|}   +(1-\lambda) M_0     \\
    &\le c_1  \lambda\max \left(   \overline a, c_2  c_3 ^{\frac {1-|\beta|}  {1-\alpha-|\beta|}}\right) ^\frac \alpha {1-|\beta|}   +(1-\lambda) M_0  \\
     & =c_4,
  \end{split}
\end{equation}
where $c_4$ is a positive constant.
We deduce from this and   \eqref{eq:10008} that
\begin{equation}
    \label{eq:10009}
  \Phi(w)\le \Phi(\one)+M_\lambda(w) \le  \Phi(\one) +c_4.
\end{equation}
From the coercivity of $\Phi$, this yields that $\max_{i} w_i < c_5$, for a positive constant $c_5$.
Then $ G_\beta(w) < c_3 $ implies that $\min_{i} w_i>\epsilon$, where $\epsilon$ is a positive constant which
can be obtained  from the exponents $\beta_i$ and the constants $c_3$ and $c_5$.
Finally, taking a smaller value of $\epsilon$ if necessary, we get      \eqref{eq:10012}.
\end{proof}

\subsubsection{Constant elasticity of substitution} \label{sec:assump_ces}
\begin{proposition}
  Assumption   \ref{ass:chap:MFG_model:technical_assump} holds with
  the  example of the  production function with the constant elasticity of substitution
described in Subsection \ref{sec:import-exampl-util} .
\end{proposition}
\begin{proof}
Combining  \eqref{eq:lambda_CES} and   \eqref{eq:target_capital_CES} implies that
\begin{equation*}
  \gamma  = \frac {\delta+\rho} \alpha  \left((\kappa^* (w))^\alpha + \sum_{j = 1}^d\left(\frac{\lambda\beta_j}{w_j}\right)^\frac{\beta_j}{1-\beta_j}\right)^{1-\gamma} (\kappa^*(w))^{1-\alpha}.
\end{equation*}
Since $\gamma\in (0,1)$, this yields
\begin{equation}
  \label{eq:10015}
    \frac {\delta+\rho} \alpha  (\kappa^*(w))^{1-\alpha +\alpha (1-\gamma)}\le \gamma.
\end{equation}
Hence $\kappa^*(w) $ is bounded from above by a positive constant independent of $w$.
From this information and the coercivity of $\Phi$, we proceed as for the Cobb-Douglas function and see that
there exists  a positive constant $c_1$ such that \eqref{eq:10004} implies that $\|w\|_{\infty}< c_1$.
  
Next, we claim that
\begin{equation}
  \label{eq:10016}
 \lim_{
 \left\{  \begin{array}[c]{l}
\min_{i=1,\dots, d} w_i \to 0,\\  \|w\|_{\infty} \le c_1
   \end{array}\right.
} g(0, w)=+\infty.     
\end{equation}
Since $g(\cdot, w)$ is non decreasing, we deduce from  \eqref{eq:10016} that there exists a constant $\epsilon>0$ independent of $\lambda$ such that \eqref{eq:10004}  implies $\min_{i} w_i >\epsilon$ and taking a smaller value of $\epsilon$ if necessary, we get  \eqref{eq:10012}.

We are left with proving \eqref{eq:10016}: we know that
\begin{displaymath}
g(k,w)\ge g(0,w)=\sup_{\ell}\left( \sum_{i} \ell_i^{\beta_i}\right)^\gamma -w\cdot \ell.  
\end{displaymath}
 A competitor can be chosen by taking  $\tilde \ell_i= w_i^{-\frac b {\beta_i}}$ where $b=\min_i \beta_i/2$. Therefore  $g(k,w)\ge \left(\sum_i w_i^{-b}\right)^\gamma -\sum_i w_i ^{1-\frac b {\beta_i}}$. The first term tends to $+\infty$ if $\min_i w_i\to 0$, while the second term is bounded since $\|w\|_\infty\le c_1$.
\end{proof}

\paragraph{\bf Acknowledgements}
All the authors  were partially supported by the ANR  (Agence Nationale de la Recherche) through MFG project ANR-16-CE40-0015-01.
Y.A. acknowledges partial support from the Chair Finance and Sustainable Development and the FiME Lab  (Institut Europlace de Finance) . The paper was completed when Y.A spent a semester at INRIA matherials. G.C. acknowledges the support of the Lagrange Mathematics and Computing Research Center.

\bibliographystyle{siamplain}
\bibliography{YA_GC_QP_DT}

\end{document}